\documentclass[a4paper,reqno,10pt]{amsart}
\usepackage{amsmath,amstext,amssymb,amsopn,amsthm,mathrsfs}
\usepackage{bbm}
\usepackage{enumerate}
\usepackage{enumitem}
\usepackage{xcolor}
\usepackage{tikz}
\usetikzlibrary{arrows}

\usepackage{hyperref}
\hypersetup{colorlinks = true, urlcolor = magenta, linkcolor = blue, citecolor = red}

\numberwithin{equation}{section}
\allowdisplaybreaks

\newtheorem{theorem}{Theorem}[section]
\newtheorem{proposition}[theorem]{Proposition}
\newtheorem{corollary}[theorem]{Corollary}
\newtheorem{lemma}[theorem]{Lemma}
\newtheorem{definition}[theorem]{Definition}
\newtheorem{remark}[theorem]{Remark}

\theoremstyle{definition}

\numberwithin{equation}{section}

\newcommand{\N}{\mathbb{N}}
\newcommand{\Z}{\mathbb{Z}}
\newcommand{\C}{\mathbb{C}}
\newcommand{\R}{\mathbb{R}^d}
\newcommand*\RR{\mathbb{R}}

\newcommand{\EEeta}{\mathcal{E}^\eta}
\newcommand{\A}{\mathcal{A}_{\eta}}
\newcommand{\sgn}{{\rm sgn}}
\newcommand{\Hom}{\eta\in{\rm Hom}(W,\,\widehat{\mathbb Z}_2)}

\newcommand{\BMO}{\mathrm{BMO}}
\newcommand*\calE{\mathcal{E}}

\newcommand*\calD{\mathcal{D}}

\newcommand*\ind{\mathbbm{1}}
\def\a{\alpha}
\def\b{\beta}
\def\v{\varphi}
\def\s{\sigma}
\def\ve{\varepsilon}

\begin{document}

\title[Hardy and  BMO spaces  on Weyl chambers]{Hardy and  BMO spaces  on Weyl chambers} 
%%%%%%%%%%%%%%%%%%%%%%%%%%%%%%%%%%%%%%%%
\author[P. Plewa]{Pawe\l{} Plewa}
\address {Wydzia\l{} Matematyki \\
         Politechnika Wroc\l{}awska\\ 
         Wyb.\,Wyspia\'nskiego 27\\ 
         50-370 Wroc\l{}aw, Poland\\   
         and Intitute of Mathematics\\
 Polish Academy of Sciences\\
and Dipartimento di Matematica\\
Politechnico di Torino} 									
\email{Pawel.Plewa@pwr.edu.pl}

\author[K. Stempak]{Krzysztof Stempak}
\address{Kie\l{}cz\'ow, Poland}      
\email{Krzysztof.Stempak@pwr.edu.pl}

%\date{\today}

%%%%%%%%%%%%%%%%%%%%%%%%%%%%%%%%%%%%%%%%%%%%%%
\begin{abstract} 
Let $W$ be a finite reflection group associated with root system $R$ in $\mathbb R^d$. Let $C_+$ denote a positive Weyl chamber distinguished by 
a choice of $R_+$, a set of positive roots. We define and investigate Hardy and BMO spaces on $C_+$ in the framework of boundary conditions 
given by  a homomorphism $\Hom$ which  attaches the $\pm$ signs to the facets of $C_+$. Specialized to orthogonal root systems, atomic decompositions 
in $H^1_\eta$ and $h^1_\eta$ are obtained and the duality problem is also treated. 
\end{abstract}

%46E30 Functional Analysis - Spaces of measurable   functions
%42B30Harmonic Analysis in several variables - H^p spaces
\thanks{Pawe{\l} Plewa acknowledges the financial support of Compagnia di San Paolo.}
\maketitle

%%%%%%%%%%%%%%%%%%%%%%%%%%%%%%%%%%%%%%%%%%%%%%

%%%%%%%%%%%%%%%%%%%%%%%%%%%%%%%%%%%%%%%%%%%%%%
\section{Introduction} \label{sec:intro}
Chang, Krantz and Stein \cite{CKS} investigated Hardy spaces on smooth domains in $\R$. Later on Chang, Dafni and Stein \cite{CDS} substantially enhanced the 
theory by introducing new distribution spaces appropriate to the Dirichlet and Neumann  problems on smooth bounded domains in $\R$. Smooth domains  
considered in \cite{CKS} included bounded Lipschitz domains, bounded domains with $C^\infty$ boundary, and (unbounded) special Lipschitz domains. For unbounded 
domains an important step toward further development of the theory of Hardy spaces was done by Auscher and Russ \cite{AR}, where strongly Lipschitz domains 
were considered (to be precise, this was done in a more general setting of elliptic second-order divergence operators). See also \cite{CMT} for the special 
Lipschitz domain context and Section \ref{sub:Hp}, where a brief account of the development of the theory of Hardy spaces on general subdomains of $\R$ is presented.

As an illustrative example, the theory of Hardy spaces on the upper half-spaces $\R_+=\{(x',x_d)\colon x'\in\mathbb{R}^{d-1},\,\, x_d>0\}$, $d\ge2$,  
was outlined in \cite[Section 1]{CKS}; see also \cite{ART} in the context of Hardy-Sobolev spaces. This open domain can be seen as a Weyl chamber 
in the framework of the (simplest possible) root system in $\R$, $R=\{-e_d,e_d\}$, $e_d=(0,\ldots,0,1)$. Consequently, the outlined theory was intimately 
connected with the relevant theory on $\R$ through the reflection $(x',x_d)\mapsto(-x',x_d)$, so that the reflection group staying behind was, up to an 
isomorphism, $\Z_2$.

The present paper includes in the investigation reflections coming from an arbitrary finite reflection group acting on $\R$. Speaking in a different way, 
apart from the groups of dilations and rotations which are naturally involved in the theory of Hardy spaces on $\R$,  we also include finite reflection 
groups through the symmetries they generate. 

In fact, our study is devoted to the theory of Hardy and  bounded mean oscillation spaces (BMO spaces for short) on Weyl chambers. Geometrically, 
for a given finite reflection group, Weyl chambers are open polyhedral cones in $\R$. They appear as connected open components of the set which emerges 
from $\R$ after removing all involved reflection hyperplanes. See Section \ref{sec:prel} for details. 

In our definition of Hardy and BMO spaces on a distinguished Weyl chamber $C_+$, we apply a procedure suggested earlier in \cite{S1,S2} by one of the authors. 
The procedure takes into account boundary conditions imposed on the facets  (flat parts of the boundary) of $C_+$. This is encoded by a homomorphism 
$\eta$ from $W$, the involved finite reflection group, into $\widehat{\mathbb Z}_2=\{-1,1\}$. In fact, our definition agrees with the commonly used definition  
of Hardy spaces associated with a given self-adjoint operator, for different classes of such operators. Namely, in the case of proposed definition, 
the resulting Hardy space $H^1_{\eta}(C_+)$ coincides with $H^1_{-\Delta^+_{\eta}}(C_+)$, where $-\Delta^+_{\eta}$ denotes the nonnegative self-adjoint 
extension of the (minus) Laplacian on $C_+$, related to the boundary conditions prescribed by $\eta$. Notably, for the distinguished homomorphisms 
${\rm triv}\equiv1$ and ${\rm sgn}={\rm det}$, $-\Delta^+_{{\rm triv}}$ and $-\Delta^+_{{\rm sgn}}$ are the Neumann and Dirichlet Laplacians 
on $C_+$, respectively, and the $\eta$-Laplacians $-\Delta^+_{\eta}$ are in between. See \cite{S2} and Section \ref{sec:semi} for details.

Our considerations  are initially situated in the setting of a general root system in $\R$. Then we narrow the investigation to orthogonal root systems. 
In this framework, where the involved geometry is simplified and the corresponding Weyl chamber is $\R_{+,k}=\RR^{d-k}\times(0,\infty)^k$, $1\le k\le d$, 
we prove atomic decompositions for functions in $H^1_\eta(\R_{+,k})$ and $h^1_\eta(\R_{+,k})$. This is the contents of Theorems \ref{thm:1}
and \ref{thm:1(loc)}; in some sense these results partly generalize those from \cite{CKS}. It is worth noting that in these theorems we include 
characterizations in terms of extensions supported in specific regions depending on $\eta$. Accordingly, we continue investigation in this framework 
discussing $\eta$-$\BMO$ spaces and proving Theorems \ref{thm:2} and \ref{thm:2(loc)}, where also intrinsic characterizations of $\eta$-$\BMO$ spaces 
are included. It is interesting to observe a natural duality between obtained results for $\eta$-Hardy and $\eta$-$\BMO$ spaces; see, for instance 
Theorems \ref{thm:1} and \ref{thm:2}, Theorems \ref{thm:1(loc)} and \ref{thm:2(loc)}, and Propositions \ref{prop:rel} and \ref{prop:relBMO}. 
Finally, all these results are then braced together in Theorems \ref{thm:3} and \ref{thm:3(loc)}, where relevant duality results are proved. We also verify that the introduced spaces are distinct for different $\eta$'s. It is also worth mentioning that for the two distinguished homomorphisms, ${\rm triv}$ and ${\rm sgn}$, 
the resulting spaces coincide with the well known \textit{extension by zero spaces} $H^1_z(\R_{+,k})$, $h^1_z(\R_{+,k})$, $\BMO_z(\R_{+,k})$ and 
${\rm bmo}_z(\R_{+,k})$, and the \textit{restriction spaces} $H^1_r(\R_{+,k})$, $h^1_r(\R_{+,k})$, $\BMO_r(\R_{+,k})$ and 
${\rm bmo}_r(\R_{+,k})$, respectively; see Corollaries \ref{rem:com2} and \ref{rem:com_bmo}.

The paper is organized as follows. In Section \ref{sec:prel} preliminaries on finite reflection groups and associated concepts are gathered. At the 
beginning of Section \ref{sec:semi} a general procedure of defining $\eta$-function/distribution spaces on Weyl chambers is described and then 
$\eta$-Hardy spaces are discussed. In Section \ref{sec:at} we focus on orthogonal root systems and prove that the functions from the $\eta$-Hardy spaces 
$H^1_\eta(\R_{+,k})$ and $h^1_\eta(\R_{+,k})$ possess atomic decompositions. Section \ref {sec:BMO} is devoted to introduction and investigation 
of the $\eta$-$\BMO$ spaces, $\BMO_\eta(\R_{+,k})$ and ${\rm bmo}_\eta(\R_{+,k})$, and then duality results are discussed. 

%%%%%5%%%%%%%%%%%%%%%%%%%%%%

\textbf{Notation and terminology}. The spaces $L^p(C_+)$, $0< p\le\infty$, are considered with respect to Lebesgue measure on $C_+$; all functions are Lebesgue 
measurable and complex-valued. By $\ind_A$ we denote the characteristic function of a (measurable) subset $A\subset\R$. The symbol $\langle\cdot,\cdot\rangle$ means the inner product in $\R$, but it is  also used for pairing functionals and testing functions; this should not lead to a confusion. Cubes are always the Euclidean open cubes with sides  parallel to the coordinate axes and $l(Q)$ stands for the sidelength of a cube $Q$. $e_1,\ldots,e_d$ denote the unit vectors of coordinate axes in $\R$. Saying that a function $F$ on $\R$ is supported in a  subset $A\subset\R$ obviously means that $F=0$ a.e. on the complement $A^c$.
When writing estimates, for nonnegative $X$ and $Y$ we will frequently use the notation $X \lesssim Y$ to indicate that $X \le C Y$ with a positive constant 
$C$ independent of significant quantities. We shall write $X \simeq Y$ when simultaneously $X \lesssim Y$ and $Y \lesssim X$. If $N,M$ are linear spaces 
with $N\subset M$, and both $N$ and $M$ are equipped with topologies (notably $N$ and $M$ are normed spaces), then writing $N\hookrightarrow M$ indicates 
that the identity map ${\rm Id}\colon N\to M$ is continuous, i.e. $N$ is continuously embedded in $M$. In the normed/quasi-normed setting this simply means 
that $\|x\|_M\lesssim\|x\|_N$ for $x\in N$, and the case of contraction (when the involved $C=1$) will be signalized by writing $N\hookrightarrow_1 M$.

%%%%%%%%%%%%%%%%%%%%%%%%%%%%%%%%%%%%%%%%%%%%%
\section{Preliminaries on finite reflection groups} \label{sec:prel}
%%%%%%%%%%%%%%%%%%%%%%%%%%%%%%%%%%%%%%%%%%%%%

Let  $R$ be a root system in the Euclidean space $\R$, that is a finite set of nonzero vectors (called roots) such that for every $\a\in R$ we have 
$\s_\a(R)=R$, where
$$
\s_\a(x)=x-\frac{2\langle \a,x\rangle}{\langle \a,\a\rangle}\a,  \qquad x\in\R,
$$
is the orthogonal reflection in $\langle \a\rangle^\bot$, the reflection hyperplane orthogonal to $\a$. The dimension of ${\rm span}\,(R)$ is called the 
\textit{rank} of $R$ and is denoted ${\rm rank}\,(R)$. If, in addition, for every $\a\in R$ we have $R\cap\mathbb{R}\a=\{\a,-\a\}$, then $R$ is called 
\textit{reduced}; throughout we assume root systems to be reduced without further mention.

The \textit{finite reflection group} $W=W(R)$ \textit{associated with} $R$ (\textit{reflection group} for short) is the subgroup of $O(\mathbb R^d)$
generated by the reflections $\s_\a$, $\a\in R$. The set $\R\setminus\bigcup_{\a\in R}\langle \a\rangle^\bot$ splits into an even number (equal to $|W|$)
of connected open components called the \textit{Weyl chambers}. $W$ acts (simply transitively) on the set of Weyl chambers and hence they are mutually
congruent. The action follows from the fact that $W$ permutes the reflection hyperplanes (for every $g\in W $, $g$ permutes the set of the reflection
hyperplanes by the rule $g\cdot\langle\a\rangle^\bot=\langle g\a\rangle^\bot$) and so, $W$ also permutes the set of connected components of
$\R\setminus\bigcup_{\a\in R}\langle \a\rangle^\bot$. A choice of $x_0\in \R$ such that  $\langle \a,x_0\rangle\neq0$ for every $\a\in R$, gives the
partition $R=R_+\sqcup (-R_+)$, where $R_+=\{\a\in R\colon \langle \a,x_0\rangle>0\}$. $R_+$ is then referred to as the \textit{set of positive roots}.
The partition distinguishes the chamber $C_+=\{x\in\R\colon \forall\, \a\in R_+ \,\,\,\langle x,\a\rangle>0\}$, which is called the \textit{positive Weyl chamber}.

Geometrically, as an intersection of a finite number of open half-spaces with supporting hyperplanes passing through the origin, $C_+$ is an \textit{open polyhedral cone} in $\R$. Notably, when $d\ge2$, $C_+$ is an example of a \textit{special Lipschitz domain}, i.e., up to a rotation, the domain above 
the graph of a Lipschitz function defined on $\mathbb R^{d-1}$. A comment on rotational invariance of the definition of special Lipschitz domain is probably necessary. In most sources, e.g. \cite[p.\,304]{CKS}, the definition of special Lipschitz domain does not include a possible rotation; see, however, \cite[Section 3.3]{Stein}.
Nevertheless, it is clear that rotation does not change essential properties of such domains. The class of special Lipschitz domains is a subclass of \textit{strongly Lipschitz domains}, i.e. open connected proper subsets  of $\R$ whose boundaries are covered by a finite union of rotated graphs of Lipschitz maps at 
most one of them being unbounded. It is also worth mentioning that, equipped with Lebesgue measure, $C_+$ is a \textit{space of homogeneous type} in the sense 
of Coifman and Weiss \cite{CW} (with the family of `balls' being truncated cubes, i.e. sets of the form $Q\cap C_+$, where Q is a cube with center in $C_+$).

We now recall the concept of simple roots. The \textsl{system of simple roots} (\textit{simple system} for short), called \textit{fundamental system} 
in \cite{DX}, is the unique subset $\Sigma\subset R_+$ which is a basis of ${\rm lin}\{\a\colon \a\in R_+\}$ and each $\a\in R_+$ is a linear combination of vectors from this basis with nonnegative coefficients, see \cite{Hu}. Consequently, $|\Sigma|={\rm rank}\,(R)$ and $C_+=\{x\in\R\colon \langle x,\a\rangle>0,\,\, \a\in\Sigma\}$, and the closure $\overline{C_+}$ has exactly ${\rm rank}\,(R)$ \textit{facets}, $\overline{C_+}\cap\langle\a\rangle^\bot$, $\a\in\Sigma$ 
($\mathcal W_\a:=\langle\a\rangle^\bot$ will be called a \textit{wall} of $C_+$). These facets are closed $(d-1)$-dimensional infinite cones (not necessarily congruent) in the hyperplanes $\langle\a\rangle^\bot$ (for $d=1$ the facet is understood as the single point, the origin). If $|\Sigma|=1$, i.e. $\Sigma=\{\a\}$, then the single facet coincides with $\langle\a\rangle^\bot$, otherwise, for $|\Sigma|\ge 2$, these facets are proper cones. If $|\Sigma|=d$, i.e.  ${\rm rank}(R)=d$, then $C_+$ is a \textit{simplicial cone}. An overview of the variety of possible Weyl chambers in low dimensions can be found in \cite[Appendix]{S2}. 

${\rm Hom}(W,\,\widehat{\mathbb Z}_2)$ will stand for the group of homomorphisms from $W$ to $\widehat{\mathbb Z}_2$, where $\widehat{\mathbb Z}_2=\{1,-1\}$ 
with multiplication. The homomorphisms $\eta\equiv1$ and $W\ni g\mapsto {\rm det}\, g$ will be denoted by ${\rm triv}$ and ${\rm sgn}$, respectively. 
Heuristically, $\Hom$ serves for assigning signs to the walls of $C_+$ by setting ${\rm sign}_\eta(\mathcal W_\a):= \eta(\s_\a)$, $\a\in\Sigma$.

Perhaps the simplest example of a root system in $\R$ is furnished by an orthogonal system of vectors $E=\{v_1,\ldots,v_m\}$, $1\le m\le d$. Then $R=E\cup(-E)$ is indeed a root system with $E$ as a set of positive roots (it suffices to take $x_0=v_1+\ldots+v_m$; in fact, $E$ is a simple system). It is worth noting that
unless $m=1$, such a system is always \textit{reducible}, i.e. splits into disjoint union of nonempty mutually orthogonal sets (each such set is itself a root
system). Reducible systems are equivalently called \textit{decomposable}.

It is clear that from the geometrical point of view root systems in $\R$ leading to geometrically congruent positive Weyl chambers should be identified. It is equivalently clear that for the geometry resulting from a given root system $R$, responsible is solely the configuration of the corresponding hyperplanes $\langle \a\rangle^\bot$, $\a\in R$. To avoid repetition in labeling the hyperplanes in what follows we  use positive roots for this labeling (a choice
of a system of positive roots is immaterial for labeling). Thus, two root systems $R_1$ and $R_2$ in $\R$ are \textsl{isomorphic} provided there exists an
orthogonal mapping in $\R$ that permutes the corresponding families of hyperplanes $\{\langle \a\rangle^\bot\}_{\a\in R_{1,+}}$ and $\{\langle \a\rangle^\bot\}_{\a\in R_{2,+}}$. Then $C_{1,+}$ and $C_{2,+}$ are congruent and the groups $W(R_1)$ and $W(R_2)$ are isomorphic.

Let $R$ be an orthogonal root system in $\R$. Up to a rotation, $R$ is isomorphic to the system $\{\pm e_{j_1},\ldots,\pm e_{j_k}\}$, where 
$1\le j_1<j_2<\ldots<j_k=d$ and $1\le k\le d$. Next, up to a permutation of the axes the latter system is isomorphic to the root system 
$R_k:=\{\pm e_{d-k+1},\ldots,\pm e_d\}$. With choice $R_{k,+}:=\{e_{d-k+1},\ldots,e_d\}$ we denote the corresponding positive Weyl chamber by $\R_{+,k}$. 
Thus $\R_{+,k}=\mathbb{R}^{d-k}\times(0,\infty)^k$ and for $k=1$ we write $\R_{+}$ rather than $\R_{+,1}$ to denote the upper half-space in $\R$; for 
$d=k=1$ we simply write $\RR_+$ to denote the half-line $(0,\infty)$.

Concluding, without losing the generality, consideration of an orthogonal root system in $\R$ may be reduced to $R_k$ for some $1\le k\le d$ and to $\R_{+,k}$ 
as the corresponding positive Weyl chamber. Moreover,$\,W(R_k)\simeq\widehat{\mathbb Z}_2^k$ and the action of any 
$\ve=(\ve_j)_{j=1}^k\in \widehat{\mathbb Z}_2^k$ on $\R$ is through
$$
x\to \ve x=(x_1,\ldots,x_{d-k},\ve_1x_{d-k+1},\ldots, \ve_k x_d).
$$
Consequently, we identify ${\rm Hom} (\widehat{\mathbb Z}_2^k, \widehat{\mathbb Z}_2)$ with $\mathbb Z_2^k$, where this time, $\mathbb Z_2=\{0,1\}$ with 
addition modulo 2. In this identification a homomorphism $\eta \in \mathbb Z_2^k$ of the reflection group represented by $\widehat{\mathbb Z}_2^k$ into 
$\widehat{\mathbb Z}_2$ acts  through $\ve \to \ve^\eta:=\prod_{j=1}^k \ve_j^{\eta_j}$ for  $\ve=(\ve_j)_{j=1}^k$. The trivial homomorphism represented 
by $(0,\ldots,0)$, which in the general case is denoted by ${\rm triv}$, in this particular case of orthogonal root systems will be denoted by ${\bf 0}$. 
On the other hand, the homomorphism represented by $(1,\ldots,1)$, which in the general case is denoted by ${\rm sgn}$, here will be denoted by $\textbf{1}$.

For a comprehensive treatment of the general theory of finite reflection groups we refer the reader to Humphreys \cite{Hu}, Kane \cite{Ka}, 
and Dunkl and Xu \cite[Chapter 4]{DX}

%%%%%%%%%%%%%%%%%%%%%%%%%%%%%%%%%%%%%%%%%%%
\section{Function and distribution spaces on $\overline{C_+}$} \label{sec:semi}
%%%%%%%%%%%%%%%%%%%%%%%%%%%%%%%%%%%%%%%%%%%

The essential aim of  this section is to propose a general procedure of defining $\eta$-function/distribution spaces on $\overline{C_+}$. This is done 
in the following two subsections. Although the procedure is then applied only to Hardy and BMO spaces, we find it reasonable to present the procedure in 
broader generality due to possible subsequent investigations in other function spaces frameworks. Throughout, if not otherwise stated, $R$, $W$, $R_+$, $C_+$, and $\Hom$ are fixed. 

Since $W$ acts simply transitively on the set of Weyl chambers, the following notion introduced in \cite{S1}, makes sense. 
Namely, a function $F$ on $\R$, identified up to a set of Lebesgue measure zero, is called $\eta$-\textit{symmetric} provided 
$$
F(gx)=\eta(g)F(x), \qquad  g\in W,\,\,  x\in\R. 
$$
Equivalently, $F$ is $\eta$-symmetric if and only if $F=\A F$, where $\A$ stands for the $\eta$-\textit{averaging operator}
$$
\A F(y)=\frac1{|W|}\sum_{g\in W}\eta(g)F(gy), \qquad y\in\R. 
$$
The $\eta$-\textit{extension operator}  
\begin{equation}\label{ext}
\EEeta f(gx)=\eta(g)f(x), \qquad x\in C_+,\quad g\in W, 
\end{equation}
extends functions on $C_+$, possibly also identified up to a set of Lebesgue measure zero on $C_+$, to $\eta$-symmetric functions on 
$\R\setminus\bigcup_{\a\in R_+}\langle \a\rangle^\bot$, or on $\R$ if we identify functions on $\R$ up to a set of Lebesgue measure zero. Notice the useful equality
\begin{align}\label{us}
\int_{\R}\EEeta f\cdot F=|W|\int_{C_+}f\cdot (\A F)|_{C_+},
\end{align}
for suitable functions $f$ and $F$.

The concepts of $\eta$-averaging and $\eta$-extension operators were introduced  in \cite{S1} and concerned functions identified up to a set of Lebesgue 
measure zero. When we wish to extend continuous functions on $\overline{C_+}$, i.e. if we want to replace  $C_+$ by $\overline{C_+}$ in \eqref{ext}, 
then we use a well known fact that the orbit of any $x\in \overline{C_+}\setminus C_+$ under the action of $W$ meets $\overline{C_+}\setminus C_+$ only at $x$. So, by \eqref{ext}, $\EEeta f$ is well defined on $\R$ for $f$ given on $ \overline{C_+}$. See Section \ref{sub:Schwartz}, where this remark is applied.

It is clear that $\A$ maps the spaces $\mathcal{S}(\R)$ and $\mathcal{D}(\R)$, respectively,  continuously into itself. 
Hence we can extend the action of $\A$ onto $\mathcal{S}'(\R)$ or $\mathcal{D}'(\R)$ by 
$$
\langle \A T,\v\rangle=\langle T,\A\v\rangle,
$$
where $T\in \mathcal{S}'(\R)$ and $\v\in \mathcal{S}(\R)$, and analogously in the other case. Obviously, this action is also continuous.  
$T\in \mathcal{S}'(\R)$ or $T\in \mathcal{D}'(\R)$ is called $\eta$-\textit{symmetric} provided  $T=\A T$.

We are now ready to propose the following procedure. Let $\mathbb X$ be a Banach, quasi-Banach or topological vector space of
functions, distributions or tempered distributions on $\R$. By $\mathbb X_\eta$ we shall denote the subspace of $\mathbb X$
consisting of all $\eta$-symmetric functions/distributions/tempered distributions in $\mathbb X$ with the norm/quasi-norm/family
of semi-norms, or just the topology, inherited from that in $\mathbb X$. It frequently happens that $\mathbb X_\eta$ is closed
in $\mathbb X$; this can be always verified by direct checking.

%%%%%%%%%%%%%%%%%%%%%%%%%%%%%%%%%%%%%%%%%%%%
\subsection{General procedure: function spaces} \label{sub:pro}

If $\mathbb X$ is a vector space of functions on $\R$, then $\mathbb X_\eta$ can be also seen as the space of restrictions to $\overline{C_+}$
of $\eta$-symmetric functions from $\mathbb X$. Equivalently, a function $f$ living on $\overline{C_+}$ belongs to $\mathbb X_\eta$
if and only if $\EEeta f$ belongs to $\mathbb X$.  Summarizing, we identify, as sets, $\mathbb X_\eta$ which is a subspace of $\mathbb X$ (of functions on
$\R$), with the following space of functions on $\overline{C_+}$,
$$
\{f\colon f\,\,{\rm is}\,\,{\rm a}\,\,{\rm function}\,\,{\rm on}\,\,\overline{C_+} \,\,{\rm and}\,\,\EEeta f\in\mathbb X\}
$$
and then we set $\|f\|_{\mathbb X_\eta}:=\|\EEeta f\|_{\mathbb X}$.
If functions in $\mathbb X$ are identified up to sets of Lebesgue measure zero,  we write $C_+$ rather than $\overline{C_+}$.

Of course in some cases the procedure does not bring in anything new in the following sense. Namely, assume that there is a well established theory
of a `category' of spaces for any open subset of $\R$ (e.g. category of $L^p$ spaces). It may happen that for any $f$ on $C_+$ and any $\eta$, the 
$\eta$-extension $\EEeta f$, as the function on $\R$,  belongs to this category space. For instance, this happens for $\mathbb X=L^p(\R)$, $0<p\le \infty$; 
we have then $\mathbb X_\eta=L^p(C_+)$ with norms that differ by the multiplicative constant $|W|$, $\|\cdot \|_{L^p(\R)_\eta}=|W|\|\cdot \|_{L^p(C_+)}$. 
(On the scale of Sobolev spaces, for $\mathbb X=W^{1,p}(\R)$ or $\mathbb X=W^{1,p}_0(\R)$, $1\le p< \infty$, it holds: for $\eta={\rm triv}$, 
$W^{1,p}(\R)_{\rm triv}=W^{1,p}(C_+)$, whereas for $\eta={\rm sgn}$, $W^{1,p}_0(\R)_{{\rm sgn}}=W^{1,p}_0(C_+)$, with equivalence of norms. 
Both results are nontrivial, see \cite{S1}.)

One more comment is necessary in case when $\mathbb X$ is a Banach space of functions on $\R$ identified modulo constants (e.g. BMO spaces). Since
$\EEeta \mathbbm{1}_{C_+}$ is not a constant function unless $\eta=\rm triv$, for general $\eta$ we remain with the definition of $\mathbb X_\eta$
as the subspace of $\mathbb X$. To be precise,  $\mathbb X_\eta$ consists of these abstract classes $[f]\in \mathbb X$ such that there exists a
representative $f_0\in [f]$ which is an $\eta$-symmetric function. The norm in $\mathbb X_\eta$ is inherited from that in $\mathbb X$.

It is desirable to provide intrinsic characterizations of the spaces $\mathbb X_\eta$, identified as spaces of functions (or functions modulo constants)
or distributions on $\overline{C_+}$. This means to impose necessary and sufficient conditions on $f$, a function or a distribution on  $\overline{C_+}$,
in terms of some objects, equivalent to the condition $ \EEeta f\in\mathbb X$. We do this in several occurrences, cf. Theorems \ref{thm:1} and \ref{thm:1(loc)}.

%%%%%%%%%%%%%%%%%%%%%%%%%%%%%%%%%%%%%%%%%%%%
\subsection{General procedure: distribution spaces} \label{sub:Schwartz} 
According to the observations made in the beginning of Section \ref{sub:pro} we first identify $\mathcal{S}_\eta(\R)$, 
the (closed) subspace of $\mathcal{S}(\R)$ of $\eta$-symmetric Schwartz functions on $\R$, with
$$
\mathcal{S}_\eta(\overline{C_+})=\{f:\overline{C_+}\to \C\colon \EEeta f\in\mathcal{S}(\R)\}.
$$
This space can be seen as the space of $\eta$-\textit{Schwartz functions} on $\overline{C_+}$. The topology in $\mathcal{S}_\eta(\overline{C_+})$, 
inherited from $\mathcal{S}(\R)$, has an intrinsic description as the topology generated by the family of seminorms 
$$
p_{\a,\b}(f)=\sup_{x\in C_+}|x^\a\partial_{\b} f(x)|, \qquad \a,\b\in\N^d.
$$
Indeed, $p_{\a,\b}(f)\le p_{\a,\b}(\EEeta f)$, where the latter is understood with the supremum taken over $\R$. In the opposite direction,
it is easily seen that $p_{\a,\b}(\EEeta f)$ is dominated, up to a multiplicative constant, by a finite sum of $p_{\a_i,\b_i}(f)$.
With this topology $\mathcal{S}_\eta(\overline{C_+})$  becomes a Fr\'echet space. 

Obviously, $\mathcal{S}_\eta(\overline{C_+})$ is a subspace of $\mathcal C^\infty(\overline{C_+})$, the vector space of $\mathcal C^\infty$ functions 
on $\Omega$ such that for any multi-index $\a$, $\partial^\a f$ extends (uniquely) to a continuous function on $\overline{\Omega}$ (this is equivalent 
to the statement that for any $\a$, $\partial^\a f$ is uniformly continuous on bounded subsets of $\Omega$). Moreover, the boundary behavior of functions 
from $\mathcal{S}_\eta(\overline{C_+})$ depends on $\eta$. Namely, with $\mathcal W_\a=\langle\a\rangle^\bot$, $\a\in\Sigma$, as the walls of $C_+$ with 
attached signs, i.e.  ${\rm sign}_\eta(\mathcal W_\a)=\eta(\s_\a)$, it is easily seen that for $f\in\mathcal{S}_\eta(\overline{C_+})$ we have the following: 
if ${\rm sign}_\eta(\mathcal W_\a)=-1$, then $f$ vanishes on the facet 
$\overline{C_+}\cap\langle\a\rangle^\bot$, whereas for ${\rm sign}_\eta(\mathcal W_\a)=1$, $\frac{\partial f}{\partial\vec{n}}$ vanishes on the $(d-1)$-dimensional interior of $\overline{C_+}\cap\langle\a\rangle^\bot$. The conclusion is that the spaces $\mathcal{S}_\eta(\overline{C_+})$ 
labeled by $\Hom $ are pairwise different.

Next we define $\mathcal{S}_\eta'(\overline{C_+})$ as the space of continuous linear functionals on $\mathcal{S}_\eta(\overline{C_+})$ with topology induced 
by the family of mappings $T\mapsto \langle T,\v\rangle$, $\v\in \mathcal{S}_\eta(\overline{C_+})$, and call it the space of $\eta$-(\textit{tempered}) 
\textit{distributions on} $\overline{C_+}$. With $\mathcal{S}_\eta(\overline{C_+})$ as a subspace of $\mathcal{S}(\R)$, it is obvious that $\mathcal{S}'(\R)\hookrightarrow \mathcal{S}_\eta'(\overline{C_+})$ (a continuous embedding).

Following the procedure described in the lines preceding Subsection \ref{sub:pro}, if $\mathbb X$ is a subspace of the space of tempered distributions on $\R$ 
(the case of distributions can be treated analogously), i.e. $\mathbb X\subset \mathcal{S}'(\R)$, then $\mathbb X_\eta$ is defined to consist of 
$\eta$-symmetric members of $\mathbb X$. But $\mathbb X_\eta$ can be also identified with a subspace of $\mathcal{S}_\eta'(\overline{C_+})$, namely with
$$
\{T\in \mathcal{S}_\eta'(\overline{C_+})\colon \EEeta T\in\mathbb X\},
$$
where the extension operator on $\eta$-tempered distributions is $\langle\EEeta T, \varphi\rangle:=\langle T, \mathcal{A}_{\eta}\varphi\rangle$ 
for $T\in \mathcal{S}_\eta'(\overline{C_+})$ and $\varphi\in \mathcal{S}(\R)$ (and $\mathcal{A}_{\eta}\varphi$ is identified with its restriction 
to $\overline{C_+}$). Moreover, with this identification, the topology in $\mathbb X_\eta$ inherited from $\mathbb X$ agrees with that in 
$\{T\in \mathcal{S}_\eta'(\overline{C_+})\colon \EEeta T\in\mathbb X\}$ 
inherited from $\mathcal{S}_\eta'(\overline{C_+})$. In what follows we shall use this parallel description of distributional $\mathbb X_\eta$.  

%%%%%%%%%%%%%%%%%%%%%%%%%%%%%%%%%%%%%%%%%%%%
\subsection{$\eta$-Hardy  spaces  on $\overline{C_+}$} \label{sub:Hp}
We commence with presenting a brief account of the development of the theory of Hardy spaces on subdomains of $\R$. 

Jonsson,  Sj\"ogren and  Wallin \cite{JSW} provided a construction of (global) Hardy spaces $H^p$ and (local) Hardy spaces $h^p$ for suitable closed subsets 
of $\R$. Additionally, the duals of these spaces were described as the homogeneous and inhomogeneous Lipschitz spaces, respectively. Miyachi \cite{M} introduced 
$H^p$ spaces on arbitrary open subsets of $\R$ together with identification of their duals. The initial definition of these spaces was given in terms of 
maximal functions. Then an atomic decomposition was established. Chang, Krantz and  Stein \cite{CKS} considered several versions of $H^p$ and $h^p$ spaces 
over open bounded subsets with smooths boundary (mostly Lipschitz). Results of \cite{CKS} were partially significantly improved in \cite{CDS} where two 
versions of $h^p$ spaces over open bounded subsets with smooth boundary (mostly $\mathcal C^\infty$) were discussed. In one of these two versions, as the 
space of test functions, the subspace of $\mathcal C^\infty(\overline{\Omega})$ with functions vanishing on the boundary was chosen. Therefore, in some sense, this part of the theory concerned closed domains. Chang \cite{C} established results on the dual spaces of local Hardy spaces discussed in \cite{CDS}. See also the paper by Auscher and Russ \cite{AR} (or rather the first (v1) arXiv version of  \cite{AR}, in the sequel for brevity denoted \cite{AR}-arXiv), where a theory of global and local Hardy spaces (as well as BMO spaces) on strongly Lipschitz domains of $\R$ is presented. We refer the reader to Stein \cite[Chapter III]{St} 
and Semmes \cite{Sem}, for a comprehensive treatment of the  theory of Hardy spaces.

Following the  procedure described above and comments made in Subsections \ref{sub:pro} and \ref{sub:Schwartz} we consider the spaces $H^p_\eta(\overline{C_+})$ and $h^p_\eta(\overline{C_+})$, $0<p\le1$. Here the `initial spaces' $H^p(\R)$ and $h^p(\R)$ are equipped with (quasi) norms $\|\cdot\|_{H^p(\R)}$ and 
$\|\cdot\|_{h^p(\R)}$, given by the grand maximal function and its truncated version,  respectively, as described in \cite[Chapter III]{St}. (This requires choosing an appropriate finite collection of seminorms on $\mathcal S(\R)$ but the resulting space is independent of this choice with equivalence of norms.) 
It is routine to check that if $\mathbb X$ is one of the latter spaces, then $\mathbb X_\eta$ is a closed subspace of $\mathbb X$. In this way, for $p=1$, 
$H^1_\eta(C_+)$ and $h^1_\eta(C_+)$ are Banach spaces of functions on $C_+$, while for $0<p<1$, $H^p_\eta(\overline{C_+})$ and $h^p_\eta(\overline{C_+})$, 
are quasi-Banach spaces of $\eta$-(tempered) distributions on $\overline{C_+}$. 

It is worth pointing out that in particular cases of a root system $R$ and a homomorphism $\eta$ these spaces appeared in the literature. 
For example, for the half-space  $\R_+$ as a Weyl chamber corresponding to the root system $R_1=\{-e_d,e_d\}$ and $\eta=\rm triv$
or $\eta={\rm sgn}$, for $p=1$ the corresponding spaces were denoted  in \cite{CKS}, as $H^1_e(\R_+)$ and $H^1_o(\R_+)$, respectively. 

Let us mention that in the case $p=1$ the following two concepts of $H^1$ spaces, as well as $h^1$ spaces, on an arbitrary open subset $\Omega$ 
of $\R$ were fundamental; namely, the concepts of \textit{restriction space} (to $\Omega$) and \textit{extension space} (by zero outside $\Omega$). Thus, 
\begin{itemize}
	\item $H^1_r(\Omega)$ is understood as the space of  restrictions to $\Omega$ of functions from $H^1(\R)$ with norm 
	$\|f\|_{H^1_r(\Omega)}:=\inf \|F\|_{H^1(\R)}$, where  $f=F|_{\Omega}$;
		\item $H^1_z(\Omega)$ is understood as the space of functions on $\Omega$ such that their extensions by zero outside $\Omega$ are in 
		$H^1(\R)$, with norm $\|f\|_{H^1_z(\Omega)}:=\|F\|_{H^1(\R)}$, where $F$ is the aforementioned extension: $F=f$ on $\Omega$ and $F=0$ on $\Omega^c$.
\end{itemize}
Analogous definitions of $h^1_r(\Omega)$ and $h^1_z(\Omega)$ follow. Obviously $H^1_z(\Omega)\hookrightarrow_1 H^1_r(\Omega)$ and $h^1_z(\Omega)\hookrightarrow_1 h^1_r(\Omega)$. Moreover, both inclusions are strict. For the first one, every function $f\in H^1_z(\Omega)$ satisfies $\int_\Omega f=0$, whereas this may not happen for $f\in H^1_r(\Omega)$; for the second one, see \cite[Proposition 6.4]{CDS}. 

When it comes to a comparison of local and global Hardy spaces then, clearly, $H^p(\R)\hookrightarrow_1h^p(\R)$. Consequently, we have 
$H^1_z(C_+)\hookrightarrow_1h^1_z(C_+)$, $H^1_r(C_+)\hookrightarrow_1 h^1_r(C_+)$, and $H^p_\eta(C_+)\hookrightarrow_1 h^p_\eta(C_+)$. 

Summing up, the theories of Hardy spaces on open subsets of $\R$ are well established. We mention that the differential operator staying behind 
these theories is the Laplacian; more precisely, including boundary conditions, the Neumann or the Dirichlet Laplacian. 

An important line of generalizations of the classical theory of Hardy spaces (and bounded mean oscillation spaces) on subdomains of $\R$ is devoted 
to some classes of nonnegative self-adjoint operators. For instance, Auscher and Russ \cite{AR} investigated Hardy spaces for divergence operators 
on special Lipschitz domains of $\R$; Duong and Yan \cite{DY} considered Hardy and BMO spaces associated with operators with heat kernel bounds. 
See the references in  \cite{AR} and \cite{DY} for further results.

Straightforward adaptations of these theories to sets which are Weyl chambers do not include intimately associated symmetries. One of the aim of 
this paper is to bring into light the important role of these symmetries.

Let $\{p_t\}_{t>0}$ be the Gauss-Weierstrass kernel on $\R$. We consider the maximal operators
$$
\mathcal M T(x)=\sup_{t>0}|\langle T,p_t(x-\cdot)\rangle|,\qquad T\in\mathcal S'(\R), \quad x\in \R,
$$
$$
mT(x)=\sup_{0<t<1}|\langle T,p_t(x-\cdot)\rangle|,\qquad T\in\mathcal S'(\R), \quad x\in \R.
$$
In particular, for $T=F\in L^1(\R)$, $\mathcal M F(x)=\sup_{t>0}|p_t*F(x)|$ is the \textit{vertical maximal function} and similarly for $mF$. 

Analogously, let $\{P_t\}_{t>0}$ be the Poisson kernel on $\R$, $P_t(x)=c_nt/(t^2+|x|^2)^{(n+1)/2}$, and let $\mathcal M^* F$ and $m^*F$ 
be the corresponding maximal operators with replacement of $p_t$ by $P_t$, for $F\in L^1(\R)$. 

It is a basic fact in the theory of Hardy spaces on $\R$ that  for $0<p\le 1$,  $\mathcal M$ and $m$ give rise to the norms/quasi-norms 
$$
\|T\|_{H^p,{\rm max}}:=\|\mathcal M T\|_p\,\,\,\quad {\rm and}\,\,\, \quad \|T\|_{h^p,{\rm max}}:= \|m T\|_p,
$$
on $H^p(\R)$ and $h^p(\R)$, which are equivalent with the norms/quasi-norms $\|\cdot\|_{H^p}$ and $\|\cdot\|_{h^p}$, respectively. 

In \cite{S1} (see also \cite{S2} and \cite{MS}) one of the authors introduced a family $\{-\Delta^+_\eta\}$, $\Hom$, of (pairwise different) nonnegative 
self-adjoint extensions of the (minus) Laplacian $-\Delta_{C_+}$ considered with $C^\infty(C_+)$ as the initial domain. The imposed boundary conditions are 
given by setting the signs $\eta(\s_\a)$ on the facets $\overline{C}_+\cap\mathcal W_\a$  of $C_+$, $\a\in\Sigma$. 

It was proved in \cite[Corollary 1.2]{S1} (see also \cite[Proposition 2.1]{S2}) that the $\eta$-heat kernel on $C_+$, i.e. the integral kernels of 
the semigroup $\{e^{-t(-\Delta^+_\eta)}\}_{t>0}$ are
\begin{equation}\label{first}
p_t^{\eta,\,C_+}(x,y)=\sum_{g\in W}\eta(g)p_t(gx-y),\qquad x,y\in C_+,\quad t>0.
\end{equation}
Similarly for  the semigroup $\{e^{-t(-\Delta^+_\eta)^{1/2}}\}_{t>0}$, the integral kernels are 
\begin{equation}\label{second}
P_t^{\eta,\,C_+}(x,y)=\sum_{g\in W}\eta(g)P_t(gx-y),\qquad x,y\in C_+,\quad t>0
\end{equation}
(see \cite[Theorem  1.1]{S1}). Obviously, \eqref{first} and \eqref{second} allow to extend $p_t^{\eta,\,C_+}$ and $P_t^{\eta,\,C_+}$ to $\R\times\R$ 
and frequently we shall consider $p_t^{\eta,\,C_+}$ and $P_t^{\eta,\,C_+}$ in this way. Observe that $p_t^{\eta,\,C_+}(x,y)$ is symmetric in $x$ and $y$, 
and for a given $x\in\R$, $p_t^{\eta,\,C_+}(x,\cdot)=|W|\A\big(p_t(x-\cdot)\big)$ and analogously for $P_t$.

We define the corresponding maximal operators (note that for every $x\in\R$, $p_t^{\eta,\,C_+}(x,\cdot)\in \mathcal S_\eta(\R))$,
\begin{equation} \label{M}
\mathcal M_{\eta} T(x)=\sup_{t>0}\Big|\langle T,p_t^{\eta,\,C_+}(x,\cdot)\rangle\Big|,\qquad T\in\mathcal S_\eta'(\R), \qquad x\in C_+,
\end{equation}
and
\begin{equation} \label{m}
m_{\eta} T(x)=\sup_{0<t<1}\Big|\langle T,p_t^{\eta,\,C_+}(x,\cdot)\rangle\Big|, \qquad T\in\mathcal S_\eta'(\R), \quad x\in C_+.
\end{equation}

It is now convenient to separate the cases $p=1$ and $0<p<1$. For $p=1$ let
$$
H^1_{-\Delta^+_\eta}(C_+)=\{f\in L^1(C_+)\colon \|f\|_{H^1_\eta(C_+),\,{\rm max}}:=\sup_{t>0}\big|e^{-t(-\Delta^+_\eta)}f\big|\in L^1(C_+)\}
$$
and
$$
H^1_{(-\Delta^+_\eta)^{1/2}}(C_+)=\{f\in L^1(C_+)\colon \|f\|_{H^1_\eta(C_+),\,{\rm max}^*}:=\sup_{t>0}\big|e^{-t(-\Delta^+_\eta)^{1/2}}f\big|\in L^1(C_+)\},
$$
with analogous definitions of $h^1_{-\Delta^+_\eta}(C_+)$ and $h^1_{(-\Delta^+_\eta)^{1/2}}(C_+)$, where taking the supremum reduces to $0<t<1$. 
%%%%%%%%%%%%%%%%%%%%%%%%%%%%%%%%%%%%%%%%%%%
\begin{proposition} \label{pro:0}
The spaces $H^1_{-\Delta^+_\eta}(C_+)$ and $H^1_{(-\Delta^+_\eta)^{1/2}}(C_+)$ coincide with $H^1_\eta(C_+)$ with equivalence of norms. 
Analogous statement holds for the local spaces $h^1_{-\Delta^+_\eta}(C_+)$ and $h^1_{(-\Delta^+_\eta)^{1/2}}(C_+)$.   
\end{proposition}
%%%%%%%%%%%%%%%%%%%%%%%%%%%%%%%%%%%%%%%%%%%
\begin{proof}
For $p=1$ and $f\in L^1(C_+)$, \eqref{M} takes the form 
$$
\mathcal M_{\eta} f(x)=\sup_{t>0}\Big|\int_{C_+}p_t^{\eta,\,C_+}(x,y)f(y)dy\Big|=\sup_{t>0}\big|e^{-t(-\Delta^+_\eta)}f(x)\big|, \qquad x\in C_+.
$$ 
Recall that 
$$
H^1_\eta(C_+)=\{f\in L^1(C_+)\colon \EEeta f\in H^1(\R)\}
$$
with norm $\|f\|_{H^1_\eta(C_+)}=\|\EEeta f\|_{H^1(\R)}$ and similarly for $h^1_\eta(C_+)$. 

It is easily seen (cf. \eqref{us}) that 
\begin{equation}\label{ble}
\int_{C_+}p_t^{\eta,\,C_+}(x,y)f(y)dy=\int_{\R}p_t(x-y)\EEeta f(y)dy, \qquad x\in \R,
\end{equation}
and hence
$$
\mathcal M_{\eta} f(x)=\mathcal M(\EEeta f)(x), \qquad x\in \R.
$$
Since $\mathcal E^\eta f$ is $\eta$-symmetric on $\R$, it follows that
$$
\|\mathcal M_{\eta} f\|_{L^1(C_+)}=\frac1{|W|}\|\mathcal  M(\EEeta f)\|_{L^1(\R)},
$$
and hence
$$
H^1_\eta(C_+)=\{f\in L^1(C_+)\colon \mathcal M_{\eta} f\in L^1(C_+)\}
$$
with norm $\|f\|_{H^1_\eta(C_+),\,{\rm max}}=\|\mathcal M_{\eta} f\|_{L^1(C_+)}$ equivalent to $\|\cdot\|_{H^1_\eta(C_+)}$. Reasoning with replacement of  
$p_t$ and $\|\cdot\|_{H^1_\eta(C_+),\,{\rm max}}$ by ${P_t}$ and  $\|\cdot\|_{H^1_\eta(C_+),\,{\rm max}^*}$, and using the maximal operator 
$\mathcal M_{\eta}^*$, is completely analogous.

Similarly,  \eqref{m} takes the form
$$
m_{\eta} f(x)=\sup_{0<t<1}\Big|\int_{C_+}p_t^{\eta,\,C_+}(x,y)f(y)dy\Big|=\sup_{0<t<1}\big|e^{-t(-\Delta^+_\eta)^{1/2}}f(x)\big|, \qquad x\in C_+,
$$
and arguments parallel  to these just used  lead to
$$
h^1_\eta(C_+)=\{f\in L^1(C_+)\colon  m_{\eta} f\in L^1(C_+)\}=\{f\in L^1(C_+)\colon  m_{\eta}^* f\in L^1(C_+)\}
$$
with norms $\|f\|_{h^1_\eta(C_+),\,{\rm max}}=\|m_{\eta} f\|_{L^1(C_+)} $ and $\|f\|_{h^1_\eta(C_+),\,{\rm max}^*}=\|m_{\eta}^* f\|_{L^1(C_+)}$ 
equivalent to $\|\cdot\|_{h^1_\eta(C_+)}$.
\end{proof}

%%%%%%%%%%%%%%%%%%%%%%%%%%%%%%%%%%%%%%%
\begin{proposition} \label{prop:Hp}
For $0<p<1$ we have
$$
H^p_\eta(\overline{C}_+)=\{T\in\mathcal S_\eta'(\R)\colon \mathcal M_{\eta} T\in L^p(C_+)\}
$$
with (quasi) norm $\|f\|_{H^p_\eta(\overline{C}_+),\,{\rm max}}:=\|\mathcal M_{\eta} T\|_{L^p(C_+)}$ equivalent to $\|\cdot\|_{H^p_\eta(\overline{C}_+)}$. 
Analogous statements hold for the local spaces $h^p_\eta(\overline{C}_+)$, $0<p<1$. 
\end{proposition}
%%%%%%%%%%%%%%%%%%%%%%%%%%%%%%%%%%%%%%%
\begin{proof}
Let $0<p<1$. Recall that, after suitable identification,  
$$
H^p_\eta(\overline{C}_+)=\{T\in\mathcal S_\eta'(\overline{C}_+)\colon \EEeta T\in H^p(\R)\},
$$
where $\langle\EEeta T,\varphi\rangle:=\langle T, \mathcal{A}_{\eta}\varphi|_{C_+}\rangle$ for $T\in \mathcal{S}_\eta'(\overline{C_+})$ and 
$\varphi\in\mathcal{S}(\R)$, with analogous equality for  $h^p_\eta(\overline{C}_+)$. 

We shall use the following distributional version of \eqref{ble},
$$
\langle T, p_t^{\eta,\,C_+}(x,\cdot)\rangle=\langle \EEeta T, p_t(x-\cdot)\rangle,\qquad T\in\mathcal S_\eta'(\overline{C}_+),\quad x\in \R;
$$
notice that the pairings refer to two different distribution spaces. With this equality, we essentially copy the reasoning from the proof of 
Proposition \ref{pro:0} to obtain the claimed result.
\end{proof}

%%%%%%%%%%%%%%%%%%%%%%%%%%%%%%%%%%%%%%%%%%%%%%%%
\section{Atomic decomposition  in $H^1_\eta$ and $h^1_\eta$ for orthogonal root systems} \label{sec:at}
%%%%%%%%%%%%%%%%%%%%%%%%%%%%%%%%%%%%%%%%%%%%%%%%
In this section we consider the orthogonal root system $R_k$, $1\leq k\leq d$, described in Section \ref{sec:prel}. Recall that in this setting 
$\R_{+,k}=\RR^{d-k}\times (0,\infty)^k$ is the corresponding positive Weyl chamber. For any $e_j\in R_{k,+}=\{e_{d-k+1},\ldots,e_d\}$ we write 
$\langle e_j\rangle^\bot_+=\{x\in\R: x_j>0 \}$ and $ \langle e_j\rangle^\bot_-=\{x\in\R: x_j< 0 \}$ for the `upper' and `lower' 
half-space with $\langle e_j\rangle^\bot$ as the supporting hyperplane. We also fix a homomorphism $\eta=(\eta_1,\ldots,\eta_k)\in\Z^k_2$  and write 
\begin{equation*}
\RR^{\eta,0}_+ = \bigcap_{i\colon\eta_i=0}\langle e_{d-k+i} \rangle^\bot_+ =\{x\in\R\colon \,\,x_{d-k+i}>0\,\,\,{\rm when}\,\,\eta_{i}=0,\,\, i=1,\ldots,k\},
\end{equation*}
$$
\RR^{\eta,1}_+ = \bigcap_{i\colon\eta_i=1}\langle e_{d-k+i} \rangle^\bot_+ =\{x\in\R\colon \,\, x_{d-k+i}>0\,\,\,{\rm when}\,\,\eta_{i}=1,\,\, i=1,\ldots,k\},
$$
with the convention that $\RR^{{\bf 1},0}_+=\RR^{{\bf 0},1}_+=\R$ (then the relevant sets of indices are empty; recall that \textbf{0} represents the trivial homomorphism, while \textbf{1} represents the sgn, i.e. the \textsl{det} homomorphism). Observe that $\RR^{{\bf 1},1}_+=\RR^{{\bf 0},0}_+=\R_{+,k}$ and 
$\RR^{\eta,0}_+ \cap \RR^{\eta,1}_+=\R_{+,k}$. The walls of $\R_{+,k}$ are $\mathcal W_i:=\langle e_{d-k+i} \rangle^\bot$, $i=1,\ldots,k$, and the signs attached by $\eta$ to these walls are ${\rm sign}_\eta(\mathcal W_i)=(-1)^{\eta_i}$. In this way the family of walls of $\RR^{\eta,0}_+$ consists of the walls $\mathcal W_i$ 
(if any) labeled by the `$+$' sign, and the family of walls of $\RR^{\eta,1}_+$ consists of these walls (if any) which are labeled by the `$-$' sign.

The essential aim of this section are Theorems \ref{thm:1} and \ref{thm:1(loc)}. To prove them we shall need some auxiliary results, Lemmas \ref{lm:1} and 
\ref{lm:2}. For brevity we state and prove the lemmas only in the `global' case, since the `local' case is analogous and the proofs require only minor 
modifications. See Remarks \ref{rm:1} and \ref{rm:2}.

In the statement of the first lemma the (classical) atomic decomposition is applied. Recall that a function  $a$ on $\R$ is a (classical) atom if $a$ is supported in a cube 
$Q$, and satisfies the cancellation condition $\int_Q a=0$ and the size condition $\Vert a\Vert_{L^2(\R)}\leq |Q|^{-1/2}$. Then $f\in H^1(\R)$ 
if and only if $f=\sum_i \lambda_i a_i$, where $a_i$ are atoms and $\{\lambda_i\}\in\ell^1$; the norm $\|f\|_{H^1(\R)}$ is comparable with the infimum of 
$\sum_i |\lambda_i|$ taken over all such decompositions. Notice that the first part of Lemma \ref{lm:1} is just a simple generalization of the known 
one-dimensional result; see \cite[Lemma~7.40]{G-CRdF}. Till the end of this section we write $\|\cdot\|_{L^2}$ rather than $\|\cdot\|_{L^2(\R)}$. 

%%%%%%%%%%%%%%%%%%%%%%%%%%%%%%%%%%%%%%%%%%%
\begin{lemma}\label{lm:1}
Let $F\in H^1(\RR^d)$ and $e$ be a coordinate vector. The following two statements hold:
	\begin{enumerate}
	\item if $F$ is even in $ \langle e\rangle^\bot$, i.e. $F=F\circ\s_e$, then $F \ind_{ \langle e\rangle^\bot_+}$ belongs to $H^1(\RR^d)$ and admits an atomic decomposition with atoms supported in $ \langle e\rangle^\bot_+$, and $\Vert F \ind_{ \langle e\rangle^\bot_+}\Vert_{H^1(\R)}\lesssim \Vert F\Vert_{H^1(\RR^d)}$;
	\item the odd extension of $F|_{ \langle e\rangle^\bot_+}$ to $\R$, namely $F\ind_{ \langle e\rangle^\bot_+}-(F\ind_{ \langle e\rangle^\bot_+})\circ\sigma_e$, is in $H^1(\R)$, and
		$$
		\Vert F\ind_{ \langle e\rangle^\bot_+}-(F\ind_{ \langle e\rangle^\bot_+})\circ\sigma_e \Vert_{H^1(\R)}\lesssim\Vert F\Vert_{H^1(\R)}.  
		$$
	\end{enumerate}
\end{lemma}
%%%%%%%%%%%%%%%%%%%%%%%%%%%%%%%%%%%%%%%%%%%
\begin{proof}
For the first part, observe that if $a$ is an  atom, then also $\widetilde{a}=\frac12(a+a\circ\s_e)\ind_{ \langle e\rangle^\bot_+}$ is an  atom. 
Indeed, let $a$ be supported in $Q$. If $Q\subset  \langle e\rangle^\bot_+$, then $\widetilde{a}= a/2$ and if  $Q\subset \langle e\rangle^\bot_-$, then 
$\widetilde{a}= a\circ\s_e/2$. In the opposite situation, when $Q\cap \langle e\rangle^\bot\neq \emptyset$, $\widetilde{a}$ is supported in a cube 
$\widetilde{Q}$ with property $|\widetilde{Q}|=|Q|$, and 
$\widetilde{Q}\subset  \langle e\rangle^\bot_+$. Moreover, in all considered cases, $\Vert \widetilde{a}\Vert_{L^2}\leq \Vert a\Vert_{L^2}$ and 
$\int_{\widetilde{Q}}\widetilde{a} =0$. Hence, $\widetilde{a}$ is an  atom supported in $ \langle e\rangle^\bot_+$ regardless of the localization of $Q$.
Let $F=\sum_{i\in\N} \lambda_i a_i$ be an atomic decomposition of $F$, even in $\langle e\rangle^\bot$. Hence,
$$
F \ind_{ \langle e\rangle^\bot_+}=\sum_{i\in\N} \lambda_i \widetilde{a_i}
$$
and the claim follows.
	
For the second statement, let $F=\sum_{i\in\N} \lambda_i a_i$ be an atomic decomposition of $F\in H^1(\R)$. Denote 
$\widetilde{F}=F\ind_{\langle e\rangle^\bot_+}-(F\ind_{ \langle e\rangle^\bot_+})\circ\sigma_e$. Clearly,
	$$
	\widetilde{F}= \sum_{i\in\N} \lambda_i \big(a_i\ind_{ \langle e\rangle^\bot_+} -(a_i\ind_{ \langle e\rangle^\bot_+})\circ\sigma_e\big).
	$$
Let $a$, supported in a cube $Q$, be one of the atoms in the decomposition of $F$. Define 
$\widetilde{a}=(a\ind_{\langle e\rangle^\bot_+}-(a\ind_{\langle e\rangle^\bot_+})\circ\sigma_e)$. If $Q\subset \langle e\rangle^\bot_-$, then $\widetilde{a}=0$. On the other hand, if $Q\subset  \langle e\rangle^\bot_+$, then $\widetilde{a}$ is just the sum of two  atoms supported in $Q$ and $\sigma_e(Q)$, respectively. It remains to consider the case when $Q$ is located on both sides of $\langle e\rangle^\bot$. Notice that then $\widetilde{a}$ is supported in a cube 
$\widetilde{Q}$ such that $Q\cup\sigma_e(Q)\subset \widetilde{Q}$ and $l(\widetilde{Q})\le 2l(Q)$. 
This implies that $|Q|\leq |\widetilde{Q}|\leq 2^d |Q|$. Therefore,
$$
\Vert \widetilde{a}\Vert_{L^2}\leq 2 \Vert a\Vert_{L^2}\leq 2|Q|^{-1/2}\leq 2^{1+d/2} |\widetilde{Q}|^{-1/2}. 
$$
Hence, $2^{-1-d/2}\widetilde{a}$ is an  atom, $\widetilde{F}$ admits an atomic decomposition, and thus it is in $H^1(\R)$. Moreover, 
$\Vert\widetilde{F}\Vert_{H^1(\R)}\lesssim\Vert F\Vert_{H^1(\R)}$. This concludes the proof.
\end{proof}

\begin{remark} \label{rm:1}
Lemma \ref{lm:1} remains valid in the `local' setting, i.e. when in the statement we replace $H^1(\RR^d)$ by $h^1(\RR^d)$, and the word `atom' by the phrase
`local atom' (we present this notion later on in this section). The proof, \textit{mutatis mutandis},  is a copy of that of Lemma \ref{lm:1}.
\end{remark}

Now we define two types of atoms, which will appear in the atomic decomposition of $H^1_\eta(\R_{+,k})$. This definition has its origins in \cite[p.~294]{CKS} 
(cf. also Section 2.1 of   \cite{AR}-arXiv).
%%%%%%%%%%%%%%%%%%%%%%%%%%
\begin{definition}\label{d1}
Let $I_0,I_1$ be two disjoint subsets of $\{1,\ldots,d\}$. Fix a function $a(x)$ supported in a cube $Q\subset\bigcap_{i\in I_0\cup I_1}\langle e_i\rangle^\bot_+$ 
and such that $\Vert a\Vert_{L^2}\leq |Q|^{-1/2}$. We say that $a$ is
	\begin{itemize}
		\item ($I_0$,$I_1$,A)-atom if $4Q\subset \bigcap_{i\in I_1} \langle e_i \rangle^\bot_+$ and $\int_Q a=0$;
		\item ($I_0$,$I_1$,B)-atom if $2Q\subset \bigcap_{i\in I_1} \langle e_i \rangle^\bot_+$ and $4Q\not\subset \bigcap_{i\in I_1} \langle e_i \rangle^\bot_+$. 
	\end{itemize}
If any of the above intersections is taken over the empty set, then it is equal to $\R$ by convention.
\end{definition}
%%%%%%%%%%%%%%%%%%%%%%%%%%%%
Notice that if $a$ is an atom of either type, then $\Vert a\Vert_{L^1(\R)}\leq 1$. Notice also that in the case $I_1=\emptyset$ there are no 
$(I_0$,$\emptyset,B)$-atoms and the ($I_0$,$\emptyset,A)$-atoms are the classical atoms in $\R$ with supporting cubes contained in 
$\bigcap_{i\in I_0}\langle e_i\rangle^\bot_+$; notably, for $I_0=\emptyset$ the ($\emptyset$,$\emptyset,A)$-atoms are the classical atoms in $\R$.

Usually, given $\eta$, we will consider $I_0=I_{\eta,0}:=d-k+\{i\colon \eta_i=0\}$ and $I_1=I_{\eta,1}:=d-k+\{i\colon \eta_i=1\}$, then saying about ($\eta,A$)- and ($\eta,B$)-atoms, respectively. Obviously, 
for $\eta={\bf 0}$, $I_{{\bf 0},0}=\{d-k+1,\ldots,d\}$, $I_{{\bf 0},1}=\emptyset$, and then $({\bf 0},A)$-atoms are just the classical  atoms supported 
in $\RR^{{\bf 0},0}_+=\R_{+,k}$, and there are no (${\bf 0},B$)-atoms. Geometrically, $(\eta,A/B)$-atoms are both supported in $\R_{+,k}$ but from 
the perspective of the larger $\RR^{\eta,1}_+ $, $(\eta,A)$-atoms are `well inside' it, while $(\eta,B)$-atoms are `boundary atoms', i.e. are relatively 
close to one of the walls of $\RR^{\eta,1}_+ $. In other words, from the perspective of $\R_{+,k}$, only the walls labeled by the `$-$' sign matter: 
$(\eta,A)$-atoms are supported relatively far from such walls, while $(\eta,B)$-atoms are supported relatively close to at least one of such  walls  
(supports of atoms of either type may touch any wall of $\R_{+,k}$ labeled by the `$+$' sign). 

In order to prove Theorem \ref{thm:1} we will consecutively apply Lemma \ref{lm:1} and the key Lemma \ref{lm:2} below. 
%%%%%%%%%%%%%%%%%%%%%%%%%%%%%%%%%%%
\begin{lemma}\label{lm:2}
	Let  $I_0$ and $I_1$ be disjoint subsets of $\{1,\ldots,d\}$, $d\ge1$, such that $|I_0\cup I_1|< d$ and let $n\in \{1,\ldots,d\}\setminus (I_0\cup I_1)$. 
	Let $F\in L^1(\R)$ be given by
\begin{equation}\label{eq:2}
	F=\sum_{i\in\N}\lambda_i a_i + \mu_i b_i\quad with\quad \sum_{i\in\N}|\lambda_i|+|\mu_i|<\infty,
	\end{equation}
	where $a_i/b_i$ are ($I_{0}$,$I_{1}$,A/B)-atoms, respectively. Then we have the decomposition 
	$$
	F\ind_{\langle e_n\rangle_+^\bot}=\sum_{i\in\N}\widetilde{\lambda_i} \widetilde{a_i} + \widetilde{\mu_i} \widetilde{b_i} \quad  with\quad  \sum_{i\in\N}|\widetilde{\lambda_i}|+|\widetilde{\mu_i}|<\infty,
	$$
where $\widetilde{a_i}/\widetilde{b_i}$ are ($I_{0}$,$I_{1}\cup\{n\}$,A/B)-atoms, respectively. Moreover,
	\begin{equation}\label{eq:5}
	\sum_{i\in\N}|\widetilde{\lambda_i}|+|\widetilde{\mu_i}|\lesssim \sum_{i\in\N}|\lambda_i|+|\mu_i|. 
	\end{equation}
\end{lemma}
%%%%%%%%%%%%%%%%%%%%%%%%%%%%%%%%%%
\begin{proof}
Fix $I_0,I_1$, $n$, and $F$ as in the statement so that $F$ is supported in $U:=\bigcap_{i\in I_0\cup I_1} \langle e_i \rangle^\bot_+$ (which is an open 
polyhedral cone with $|I_0\cup I_1|$ mutually perpendicular walls when $|I_0\cup I_1|>0$, and is equal to $\R$ when $I_0=I_1=\emptyset$). By \eqref{eq:2} we have
$$
	F\ind_{ \langle e_n\rangle^\bot_+} = \sum_{i\in\N}\lambda_i a_i\ind_{ \langle e_n\rangle^\bot_+} + \mu_i b_i\ind_{ \langle e_n\rangle^\bot_+}.
$$
	
Firstly, we consider $a_i$. For $i\in\N$ let $Q_i\subset U$ be the cube corresponding to the $(I_0,I_1,A)$-atom $a_i$. We distinguish four cases depending 
on the location of $Q_i$ with respect to $ \langle e_n\rangle^\bot_+$:
	\begin{enumerate}
		\item $4Q_i\subset  \langle e_n\rangle^\bot_+$;
		\item $2Q_i\subset  \langle e_n\rangle^\bot_+$ and $4Q_i\not\subset  \langle e_n\rangle^\bot_+$;
		\item $Q_i\cap  \langle e_n\rangle^\bot_+\neq\emptyset$ and $2Q_i\not\subset  \langle e_n\rangle^\bot_+$;
		\item $Q_i\cap  \langle e_n\rangle^\bot_+=\emptyset$.
	\end{enumerate}
Notice that the fourth possibility will not be counted, because $a_i\ind_{ \langle e_n\rangle^\bot_+}\equiv 0$. Hence, we can write
$$
\sum_{i\in\N}\lambda_i a_i\ind_{\langle e_n\rangle^\bot_+}=\sum_{i\in\N}\big(\lambda_i^1 a_i^1+\lambda_i^2 a_i^2+\lambda_i^3 a_i^3\ind_{\langle e_n\rangle^\bot_+}\big),
$$
where $a_i^m$, $m\in\{1,2,3\}$, correspond to the $m$-th case above, and hence $\{\lambda_i^1\}\cup \{\lambda_i^2\}\cup \{\lambda_i^3\}\subset \{\lambda_i\}$.
	
We repeat these steps for $b_i$ obtaining 
$$
F\ind_{ \langle e_n\rangle^\bot_+} = \sum_{i\in\N} \big( \lambda_i^1 a_i^1 + \lambda_i^2 a_i^2 + \lambda_i^3 a_i^3\ind_{\langle e_n\rangle^\bot_+}+\mu_i^1 b_i^1 + \mu_i^2 b_i^2 + \mu_i^3 b_i^3\ind_{\langle e_n\rangle^\bot_+}\big),
$$
additionally with $\{\mu_i^1\}\cup \{\mu_i^2\}\cup \{\mu_i^3\}\subset \{\mu_i\}$. Letting $I_1':=I_{1}\cup\{n\}$ observe that $a_i^1$ are $(I_{0},I_{1}',A)$-atoms and $a_i^2,b^1_i,b^2_i$ are $(I_{0},I_{1}',B)$-atoms. Thus, we are left with $a_i^3\ind_{ \langle e_n\rangle^\bot_+}$ and 
$b_i^3\ind_{ \langle e_n\rangle^\bot_+}$. We will show that the latter functions can be decomposed into a sum of ($I_{0},I_1',B)$-atoms with coefficients that sum up to a uniformly majorized constant. For this, rather than to treat each of the two cases separately, we fix a function $G$ supported in a cube 
$Q\subset U$, such that $Q\cap\langle e_n\rangle^\bot_+\neq\emptyset$, $2Q\not\subset\langle e_n\rangle^\bot_+$, and 
$\Vert G\Vert_{L^2}\leq |Q|^{-1/2}$ ($a_i^3\ind_{\langle e_n\rangle^\bot_+}$ and $b_i^3\ind_{ \langle e_n\rangle^\bot_+}$ satisfy these conditions). We will verify that these conditions are sufficient to perform the decomposition.
	
We introduce the Whitney partition of $U\cap\langle e_n\rangle^\bot_+$ (an open polyhedral cone with simultaneously perpendicular $|I_0\cup I_1|+1$ walls),  
relative to the hyperplane  $\langle e_n\rangle^\bot$,
$$
U\cap  \langle e_n\rangle^\bot_+=\bigcup_{m\in\Z}\, \bigcup_{D\in\calD_m} D, 
$$
where $\calD_m$ consists of all diadic cubes $D\subset U\cap\langle e_n\rangle^\bot_+$ with corners in $2^{-m}\cdot\Z^d$, 
such that 
$$ 
l(D)=\mathrm{dist}\,\big(D,\langle e_n\rangle^\bot\big)=2^{-m}.
$$
Notice that this means that $2D\subset  U\cap  \langle e_n\rangle^\bot_+$, but $4D\not\subset U\cap  \langle e_n\rangle^\bot_+$. 

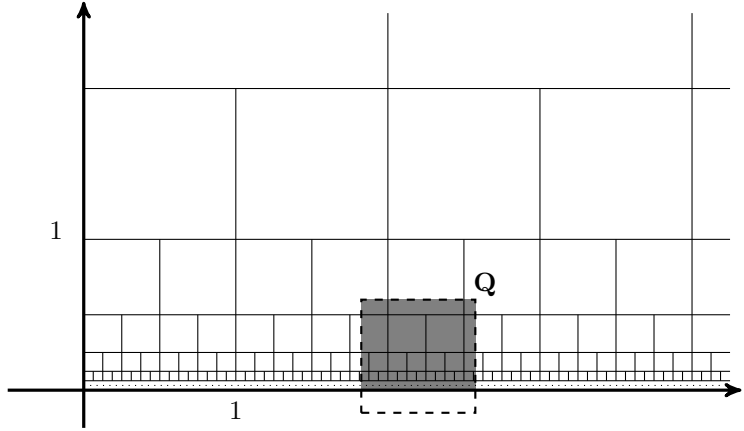
\begin{figure}[ht]
\begin{tikzpicture}
[scale=0.5, axis/.style={very thick, ->, >=stealth'}]

%\cube Q
\fill[black!50] (7.3,0) rectangle (10.3,2.4);
\draw[thick, dashed] (7.3,-0.6) rectangle (10.3,2.4);
\node[label={[above right]$\mathbf{Q}$}] at (10,2.0){} ;
%axis
\draw[axis] (-2,0)  -- (17.3,0) ;
\draw[axis] (0,-1) -- (0,10.3) ;

%1's on axis
\node[label={[below]$1$}] at (4,-0.3){} ;
\node[label={[left]$1$}] at (-0.3,4){} ;

%horizontal lines
\draw  (0,1)--(17,1);
\draw  (0,2)--(17,2);
\draw  (0,4)--(17,4);
\draw  (0,8)--(17,8);
\draw  (0,1/2)--(17,1/2);
\draw  (0,1/4)--(17,1/4);

%\vertical lines
%2^2
\draw  (4,1/4)--(4,8);
\draw  (8,1/4)--(8,10);
\draw  (12,1/4)--(12,8);
\draw  (16,1/4)--(16,10);

%2^1
\draw  (2,1/4)--(2,4);
\draw  (6,1/4)--(6,4);
\draw  (10,1/4)--(10,4);
\draw  (14,1/4)--(14,4);

%2^{0}
\draw  (1,1/4)--(1,2);
\draw  (3,1/4)--(3,2);
\draw  (5,1/4)--(5,2);
\draw  (7,1/4)--(7,2);
\draw  (9,1/4)--(9,2);
\draw  (11,1/4)--(11,2);
\draw  (13,1/4)--(13,2);
\draw  (15,1/4)--(15,2);

%2^{-1}
\draw  (1/2,1/4)--(1/2,1);
\draw  (3/2,1/4)--(3/2,1);
\draw  (5/2,1/4)--(5/2,1);
\draw  (7/2,1/4)--(7/2,1);
\draw  (9/2,1/4)--(9/2,1);
\draw  (11/2,1/4)--(11/2,1);
\draw  (13/2,1/4)--(13/2,1);
\draw  (15/2,1/4)--(15/2,1);
\draw  (17/2,1/4)--(17/2,1);
\draw  (19/2,1/4)--(19/2,1);
\draw  (21/2,1/4)--(21/2,1);
\draw  (23/2,1/4)--(23/2,1);
\draw  (25/2,1/4)--(25/2,1);
\draw  (27/2,1/4)--(27/2,1);
\draw  (29/2,1/4)--(29/2,1);
\draw  (31/2,1/4)--(31/2,1);
\draw  (33/2,1/4)--(33/2,1);

%2^{-2}
\draw  (1/4,1/4)--(1/4,1/2);
\draw  (3/4,1/4)--(3/4,1/2);
\draw  (5/4,1/4)--(5/4,1/2);
\draw  (7/4,1/4)--(7/4,1/2);
\draw  (9/4,1/4)--(9/4,1/2);
\draw  (11/4,1/4)--(11/4,1/2);
\draw  (13/4,1/4)--(13/4,1/2);
\draw  (15/4,1/4)--(15/4,1/2);
\draw  (17/4,1/4)--(17/4,1/2);
\draw  (19/4,1/4)--(19/4,1/2);
\draw  (21/4,1/4)--(21/4,1/2);
\draw  (23/4,1/4)--(23/4,1/2);
\draw  (25/4,1/4)--(25/4,1/2);
\draw  (27/4,1/4)--(27/4,1/2);
\draw  (29/4,1/4)--(29/4,1/2);
\draw  (31/4,1/4)--(31/4,1/2);
\draw  (33/4,1/4)--(33/4,1/2);
\draw  (35/4,1/4)--(35/4,1/2);
\draw  (37/4,1/4)--(37/4,1/2);
\draw  (39/4,1/4)--(39/4,1/2);
\draw  (41/4,1/4)--(41/4,1/2);
\draw  (43/4,1/4)--(43/4,1/2);
\draw  (45/4,1/4)--(45/4,1/2);
\draw  (47/4,1/4)--(47/4,1/2);
\draw  (49/4,1/4)--(49/4,1/2);
\draw  (51/4,1/4)--(51/4,1/2);
\draw  (53/4,1/4)--(53/4,1/2);
\draw  (55/4,1/4)--(55/4,1/2);
\draw  (57/4,1/4)--(57/4,1/2);
\draw  (59/4,1/4)--(59/4,1/2);
\draw  (61/4,1/4)--(61/4,1/2);
\draw  (63/4,1/4)--(63/4,1/2);
\draw  (65/4,1/4)--(65/4,1/2);
\draw  (67/4,1/4)--(67/4,1/2);

%dots
\draw[dotted] (0,1/8) -- (17,1/8);

\end{tikzpicture}
\caption{The Whitney decomposition in the case $d=2$, $U=(0,\infty)\times \RR$, and $n=2$, with an example of the cube $Q$.}
\label{fig:1}
\end{figure}

Now we write
\begin{equation}\label{new}
	Q\cap \langle e_n\rangle^\bot_+=\bigcup_{m\in\Z}\, \bigcup_{j=1}^{j_m}Q\cap  D_{mj}, 
\end{equation}
where $D_{mj}\in\calD_m$ and $Q\cap D_{mj}\neq \emptyset$. In \eqref{new} and in similar places, equality is meant up to a set of Lebesgue measure zero.
	
We now establish a bound on $j_m$ which is crucial for further development. Obviously, with given $Q$, $j_m$ depends on the choice of $I_0,I_1$ and $n$ but 
it is geometrically clear that the `worst' case is when $|I_0\cup I_1|=d-1$ (then the  number of walls of $U\cap\langle e_n\rangle^\bot_+$ is $d$). 
Hence, estimating $j_m$, we can simply assume that $U=(0,\infty)^{d-1}\times \RR$ and $n=d$, so that $U\cap\langle e_n\rangle^\bot_+=\R_+$.
We claim that for any admissible $Q$ with $2^{-m_0-1}<l(Q)\leq 2^{-m_0}$ we have
\begin{equation}\label{eq:1}
j_m\le C_d\times
\begin{cases}
2^{(m-m_0)(d-1)}, \qquad m\ge m_0,\\
 0, \,\,\,\,\,\qquad m< m_0. 
\end{cases}\\
\end{equation}
A dilation argument shows that it suffices to prove this estimate for $m_0=0$. Thus, let $Q'\subset (0,\infty)^{d-1}\times \RR$ be such that 
$Q'\cap\langle e_d\rangle^\bot_+\neq\emptyset$, $2Q'\not\subset\langle e_d\rangle^\bot_+$, $l(Q')=1$, and \eqref{new} holds. By the geometry of situation, 
it is clear that $j_m=0$ for $m<0$ and $j_m\le C_d2^{m(d-1)}$ for $m\ge0$ ($C_d=2^d-1$ suffices) which proves the claim.  
	
We now decompose $G\ind_{\langle e_n\rangle^\bot_+}$ as
$$
	G\ind_{\langle e_n\rangle^\bot_+}=\sum_{m=m_0}^\infty \sum_{j=1}^{j_m} \frac{1}{\alpha_{mj}} \big( \alpha_{mj} G\ind_{Q\cap D_{mj}}\big),
$$
where 
$$
	\alpha_{mj}=|D_{mj}|^{-1/2} \Vert G\ind_{Q\cap D_{mj}}\Vert_{L^2}^{-1}. 
$$
Observe that with this choice of $\alpha_{mj}$ the functions $G_{mj}:=\alpha_{mj} G\ind_{Q\cap D_{mj}}$ are $(I_0,I_1',B)$-atoms; this is because  
$G_{mj}$ is supported in the cube $D_{mj}$ and $\Vert G_{mj}\Vert_{L^2} =|D_{mj}|^{-1/2}$. Additionally, by using \eqref{eq:1} we obtain
	\begin{align*}
\sum_{m=m_0}^\infty \sum_{j=1}^{j_m} |\alpha_{mj}|^{-1}
&\le\Big( \sum_{m=m_0}^\infty \sum_{j=1}^{j_m} |D_{mj}|\Big)^{1/2} \Big( \sum_{m=m_0}^\infty \sum_{j=1}^{j_m}\Vert G\ind_{Q\cap D_{mj}}\Vert_{L^2}^2\Big)^{1/2}\\
&=\Big(\sum_{m=m_0}^\infty \sum_{j=1}^{j_m}|D_{mj}|\Big)^{1/2}\Vert G\ind_{\langle e_n\rangle^\bot_+}\Vert_{L^2}\\
&\le\Big(\sum_{m=m_0}^\infty\sum_{j=1}^{j_m}2^{-dm}\Big)^{1/2}\Vert G\Vert_{L^2}\\
&\leq \big(2^dC_d\big)^{1/2}.
	\end{align*}
The latter inequality is a consequence of $\Vert G\Vert_{L^2}\leq |Q|^{-1/2}$, and 
\begin{equation*}
\sum_{m=m_0}^\infty\sum_{j=1}^{j_m}2^{-dm}\le C_d 2^{-m_0(d-1)}\sum_{m=m_0}^\infty 2^{-m}\le2^dC_d|Q|.
\end{equation*}
	
Combining the above shows that $G=\sum_{k\in\N}\tau_k\hat b_k$, where $\hat b_k$ are  $(I_{0},I_{1}',B)$-atoms and $\sum_k |\tau_k|$ is bounded by a 
universal constant not depending on $G$. Thus, for every $i\in \N$, $a_i^3\ind_{\langle e_n\rangle^\bot_+}$ and $b_i^3\ind_{ \langle e_n\rangle^\bot_+}$ 
possess the analogous decompositions. This proves existence of required decomposition of $F\ind_{\langle e_n\rangle_+^\bot}$ with coefficients satisfying 
\eqref{eq:5} (which follows from observations made in between) and finishes the proof of the lemma.
\end{proof}

\begin{remark} \label{rm:2}
Lemma \ref{lm:2} remains valid in the `local' setting, i.e. when in the statement we modify the phrases `($I_{0}$,$I_{1}$,A/B)-atoms' and 
`($I_{0}$,$I_{1}\cup\{n\}$,A/B)-atoms' by adding the word `local' in front of them (we define these local analogues later on in this section, 
see Definition \ref{def:2}). The proof is a copy of that of Lemma \ref{lm:2} but we add the following comments. 

Although local $(I_0,I_1,A/B)$-atoms differ from $(I_0,I_1,A/B)$-atoms by limiting first to the case $l(Q)\le1$ and then admitting local $(I_0,I_1,B)$-atoms 
that satisfy merely the condition $l(Q)>1$ (apart of the basic conditions on the supporting cube and the $L^2$ norm), the first step, the decomposition of 
$F\ind_{ \langle e_n\rangle^\bot_+}$ into a series of summands of six types, relies on distinguishing local atoms according to the location of their supporting 
cubes. Keeping the notation from the 
proof of Lemma \ref{lm:2} we arrive to the point where $a_i^3\ind_{\langle e_n\rangle^\bot_+}$ and $b_i^3\ind_{\langle e_n\rangle^\bot_+}$ should be 
decomposed into a series of local $(I_0,I_1',B)$-atoms with coefficients that sum up to a uniformly majorized constant. Again the task reduces to considering  
a function $G$ supported in a cube $Q\subset U$, such that $Q\cap\langle e_n\rangle^\bot_+\neq\emptyset$, $2Q\not\subset\langle e_n\rangle^\bot_+$, 
and $\Vert G\Vert_{L^2}\leq |Q|^{-1/2}$, and to decompose it into a series of local $(I_0,I_1',B)$-atoms with relevant control of coefficients.
\end{remark}

We are now ready to state and prove one of the two main results of this section. Notice that implicitly Theorem \ref{thm:1} contains the statement 
that ($\eta,A/B$)-atoms belong to $H^1_\eta(\R_{+,k})$. Note, however, that for $\eta={\bf 0}$ there are no $(\textbf{0},B)$-atoms, and consequently, 
the representation of $f$ in (2) has only $(\textbf{0},A)$-atoms, i.e. $\mu_i\equiv 0$. Also, for $\eta={\bf 1}$, the assumptions 
imposed on the support of $F$ are void (recall that $\RR^{{\bf 1},0}_+=\R$ and $\RR^{{\bf 1},1}_+=\R_{+,k}$) and (3) and (4) are identical. 
Finally, observe that since $\RR^{\eta,0}_+\subset \big(\RR^{\eta,1}_+\setminus \R_{+,k}\big)^{\mathrm{c}}$, for $\eta\neq{\bf 1}$ condition (4) 
looks \textit{a priori}  weaker than condition (3) but \textit{a posteriori} both are equivalent.
%%%%%%%%%%%%%%%%%%%%%%%%%%%%%%%%%%
\begin{theorem}\label{thm:1}
	Let $f\in L^1(\R_{+,k})$ and $\eta\in\Z_2^k$. The following conditions are equivalent: 
	\begin{enumerate}
		\item $f\in H^1_\eta(\R_{+,k})$;
		\item $f=\sum_{i\in\N} \lambda_i a_i+\mu_i b_i$, where $a_i/b_i$ are ($\eta$,A/B)-atoms, respectively, and $\{\lambda_i\},\{\mu_i\}\in\ell^1(\N)$;
		\item there exists $F\in H^1(\RR^d)$, supported in $\RR^{\eta,0}_+$ and  such that $F\big|_{\R_{+,k}}=f$;
		\item there exists $F\in H^1(\RR^d)$, supported in $(\RR^{\eta,1}_+\setminus \R_{+,k})^c$ and  such that $F\big|_{\R_{+,k}}=f$.
		\end{enumerate}
	Moreover, the relevant quantities are comparable, namely
	\begin{equation}\label{eq:4}
	\Vert  f\Vert_{H^1_\eta(\R_{+,k})} \simeq \inf \sum_{i\in\N}  |\lambda_i|+|\mu_i| \simeq \inf_F \Vert F\Vert_{H^1(\R)}, 
	\end{equation}
	where the first infimum is taken over all decompositions as in (2) and the second over all $F$ as in (3) or (4).
\end{theorem}
%%%%%%%%%%%%%%%%%%%%%%%%%%%%%%%%%%%%
\begin{proof} We fix $f\in L^1(\R_{+,k})$  throughout. Recall that the walls of $\R_{+,k}$ are $\mathcal W_i:=\langle e_{d-k+i} \rangle^\bot$, $i=1,\ldots,k$, 
and ${\rm sign}_\eta(\mathcal W_i)=(-1)^{\eta_i}$. 

\textbf{(2)$\implies$(1)} Assume $f$ admits a decomposition as in (2). To check that $f\in H^1_{\eta}(\R_{+,k})$ with 
$\Vert f\Vert_{H^1_{\eta}(\R_{+,k})}\lesssim \sum_{i\in\N} \big|\lambda_i\big|+\big|\mu_i\big|$ uniformly on $f$, it suffices to verify that 
$\calE^{\eta} f\in H^1(\RR^d)$ and $\Vert \calE^\eta f\Vert_{H^1(\R)}\lesssim \sum_{i\in\N} \big|\lambda_i\big|+\big|\mu_i\big|$, again uniformly on $f$. By linearity and boundedness of $\calE^{\eta}$ from $L^1(\R_{+,k})$ to $L^1(\R)$, we have
$$
\calE^{\eta}f=\sum_{i\in\N} \lambda_i \calE^{\eta}a_i + \mu_i \calE^{\eta} b_i,
$$
so we are reduced to consider the action of $\calE^{\eta}$ on $(\eta,A/B)$-atoms and to check that $\Vert \calE^{\eta}a \Vert_{H^1(\RR^d)}\lesssim 1$ 
for $(\eta,A)$-atoms, and analogously for $(\eta,B)$-atoms.  

If $a$ is an $(\eta,A)$-atom, then checking the above claim is immediate since $\calE^\eta a$ is a sum of $2^{k}$ atoms and $\Vert \calE^{\eta}a \Vert_{H^1(\RR^d)}\lesssim 1$ follows. Now, let $b$ be an  $(\eta,B)$-atom supported in a cube $Q\subset \R_{+,k}$ (so, from now on we assume $\eta\neq\textbf{0}$). 
We will show that $\calE^\eta b$ is a sum of a number of scaled atoms with scaling constant depending only on the dimension $d$ and hence 
$\Vert \calE^{\eta}b \Vert_{H^1(\RR^d)}\lesssim 1$ will follow. This will be done by representing the extension operator $\calE^\eta$ as a composition 
of two extension operators related to contexts of two complementary subroot systems of $R_k$. 

By the definition of $(\eta,B)$-atom we have $4Q\cap\mathcal W_j\neq\emptyset$ for some $j\in\{1,\ldots,k\}$ with $\eta_j=1$. Denote by $J$ the set of all 
such $j$'s. Let $R_{J}=\{\pm e_{d-k+j}\colon j\in J\}$ with analogous definition of $R_{J^c}$, where $J^c$ stands for the complement of $J$ in $\{1,\ldots,k\}$; 
$R_{J}$ and $R_{J^c}$ are complementary root (sub)systems of $R_k$, mutually orthogonal. Then $W(R_{J})\simeq\Pi_{j\in J}\widehat{\mathbb Z}_2$ and
similarly for $W(R_{J^c})$, and $W(R_k)$ naturally identifies with the direct product $W(R_{J})\oplus W(R_{J^c})$. As a positive Weyl chamber associated 
with $R_{J}$ 
we can choose $\R_{+,J}=\mathbb{R}^{d-k}\times \Pi_{j\in J} Y_j$, where $Y_j=(0,\infty)$ for $j\in J$ and $Y_j=\RR$ otherwise, and similarly for $\R_{+,J^c}$, 
a positive Weyl chamber associated with $R_{J^c}$. Clearly,  $\R_{+,k}\subset \R_{+,J}$, the walls of $\R_{+,J}$ are $\mathcal W_j$, $j\in J$, and 
all of them are labeled by the `$-$' sign. Also $\eta$ can be viewed as the 'tensor product' of the corresponding homomorphisms on $W(R_{J})$ and $W(R_{J^c})$,
namely $\eta=\textbf{1}_J\otimes\eta^c$, where $\textbf{1}_J=(1,\ldots,1)$ ($|J|$ times) and $\eta^c$  is obtained from $\eta$ by removing all the entries of 
$\eta$ on $j$'s places, $j\in J$. Due to the product structure of $W(R_k)$ and the tensor product structure of $\eta$ it is easily seen that 
$\calE^{\eta}f=\calE^{\eta^c}(\calE^{\textbf{1}_J}f)$ for any $f$ supported in $\R_{+,k}$; notice that if $f$ is such a function, then $\calE^{\textbf{1}_J} f$ 
is supported  in $\R_{+,J^c}$ so that the composition makes sense.

Let $b_J:=\calE^{\textbf{1}_J}b=\sum_{g\in W(R_{J})} \textbf{1}_J(g) b\circ g$. Clearly, $b_J$ has mean-value zero. Denote $Q_g=g(Q)$ for $g\in W(R_{J})$ and 
choose a smallest cube $Q_J$ (in the sense of sidelenght) such that it contains each $Q_g$. Notice that the volumes of $Q$ and $Q_J$ are comparable. It 
follows that $b_J$ is a scaled atom supported in $Q_J$, where the scaling constant depends only on the dimension $d$, and so $\|b_J \|_{H^1(\RR^d)}\lesssim 1$. Finally, $\calE^\eta b=\calE^{\eta^c} b_J$ is a sum of $2^{k-|J|}$ scaled atoms and hence $\Vert \calE^{\eta}b \Vert_{H^1(\RR^d)}\lesssim 1$.

\textbf{(1)$\implies$(3)} If $f\in H^1_{\eta}(\R_{+,k})$, then $\calE^{\eta}f\in H^1(\R)$ and it is even in each $\langle e_{d-k+i}\rangle^\bot$ such that 
$\eta_i=0$. Hence, applying Lemma \ref{lm:1} (1) consecutively to each such $e_{d-k+i}$ (if any) yields $\big(\calE^{\eta}f\big)\ind_{\RR^{\eta,0}_+}\in H^1(\R)$, 
and this function has the required properties. Notably, we have
	$$
	\big\Vert\big(\calE^{\eta}f\big)\ind_{\RR^{\eta,0}_+} \big\Vert_{H^1(\R)}\lesssim \Vert \calE^{\eta}f \Vert_{H^1(\R)}=\Vert f \Vert_{H^1_\eta(\R_{+,k})}.
	$$ 
	
	\textbf{(3)$\implies$(2)} Let $F\in  H^1(\RR^d)$ be as in (3). In particular, since $F$ is supported in $\RR^{\eta,0}_+$ and the latter region is a special 
Lipschitz domain, $F$ admits a (classical) atomic decomposition whose atoms are supported in $\RR^{\eta,0}_+$; see \cite[Theorem 3.3]{CKS} or the comments 
following  \cite[Theorem 1]{AR}-arXiv and referring to the setting of strongly Lipschitz domains. In the notation of Lemma \ref{lm:2} 
this means a decomposition of the form \eqref{eq:2} into ($I_{\eta,0}$,$\emptyset$,A)-atoms (with $\mu_i\equiv0$). Now, applying  consecutively Lemma \ref{lm:2} 
to each $n\in I_{\eta,1}$ (if any), we obtain an atomic decomposition of $F\ind_{\R_{+,k}}$ into $(\eta,A)$-atoms and $(\eta,B)$-atoms. Moreover, by \eqref{eq:5} we have the relevant inequality in \eqref{eq:4}. 

\textbf{(4)$\implies$(3)} Let $F\in  H^1(\RR^d)$ be as in (4). We assume that $\eta\neq\textbf{1}$. Let $1\le\eta_{i_1}<\ldots<\eta_{i_r}\le d$, 
$1\le r\le d$, be the all entries of $\eta$ which are equal to 1. We apply Lemma \ref{lm:1} (2) to $e_{i_1}$ and $F$ to obtain  $F_1\in H^1(\R)$ 
being the odd extension of $F|_{\langle e_{i_1}\rangle_+^\bot}$. Then we continue this process to $e_{i_2}$ and $F_1$ to get the odd extension of 
$F_1|_{\langle e_{i_2}\rangle_+^\bot}$, and so on. The final function $F_r$ belongs to $H^1(\R)$ and $\|F_r\|_{H^1(\R)}\lesssim\|F\|_{H^1(\R)}$. 
Moreover, by the assumptions imposed  on $F$ and the way $F_r$ was constructed, we have $F_r=(\calE^\eta f)\ind_{\RR^{\eta,0}_+}$. 
Notably, $F_r$ is supported in $\RR^{\eta,0}_+$ and $F_r|_{\R_{+,k}}=f$. 

\textbf{(3)$\implies$(4)} This implication is trivial.
	
Finally, notice that the equivalence \eqref{eq:4} follows by collecting  the partial conclusions obtained in the proofs of the four implications. 
\end{proof}

We now pass to discussion of the local Hardy space $h^1_\eta$. Recall that a function $a$ on $\R$ is a (classical) \textit{local} atom if it is supported 
in a cube $Q$ satisfying $\Vert a\Vert_{L^2}\leq |Q|^{-1/2}$ and either $l(Q)> 1$, or $l(Q)\le1$ and $\int_Q a=0$. Then $f\in h^1(\R)$ if and only if 
$f=\sum_i \lambda_i a_i$, where $a_i$ are local atoms and $\{\lambda_i\}\in\ell^1$; the norm $\|f\|_{h^1(\R)}$ is comparable with the infimum of 
$\sum_i |\lambda_i|$ taken over all such decompositions.

We now introduce the local analogue of $(\eta,A/B)$-atoms. Again, if $a$ is a local atom of either type, then $\Vert a\Vert_{L^1(\R)}\leq 1$; 
also the case $I_0=I_1=\emptyset$ is admitted.

%%%%%%%%%%%%%%%%%%%%%%%%%%
\begin{definition}\label{def:2}
Let $I_0,I_1$ be two disjoint subsets of $\{1,\ldots,d\}$. Fix a function $a(x)$ supported in a cube $Q\subset\bigcap_{i\in I_0\cup I_1}\langle e_i\rangle^\bot_+$
and such that $\Vert a\Vert_{L^2}\leq |Q|^{-1/2}$. We say that $a$ is
\begin{itemize}
\item local $(I_0,I_1,A)$-atom if $l(Q)\le 1$, $4Q\subset \bigcap_{i\in I_1} \langle e_i \rangle^\bot_+$, and $\int_Q a=0$;
\item local $(I_0,I_1,B)$-atom if either: 1) $l(Q)>1$, or 2) $l(Q)\le1$, $2Q\subset \bigcap_{i\in I_1} \langle e_i \rangle^\bot_+$, and 
$4Q\not\subset \bigcap_{i\in I_1} \langle e_i \rangle^\bot_+$.
\end{itemize}
If $I_0=I_{\eta,0}$ and $I_1=I_{\eta,1}$ for some $\eta\in\Z_2^k$ (see the comments following Definition \ref{d1} where $I_{\eta,0}$, $I_{\eta,1}$ were defined), we refer to these atoms as to the local $(\eta,A/B)$-atoms.
\end{definition}
%%%%%%%%%%%%%%%%%%%%%%%%%%%%

Notice that for $\eta={\bf 0}$ there exist local $(\textbf{0},B)$-atoms (those with `large supports'), unlike in the global case. In fact, we will avoid 
saying `large' or `small' supports in the future, because a function $a$ supported in a `small' cube $Q$, $l(Q)<1$, which is placed anywhere in $\RR_{+,k}^d$, 
is a local $(\eta,B)$-atom provided that $\Vert a\Vert_{L^2}\leq1$, since then it can be viewed as supported in a `large' cube. Moreover, mind that for such 
an $a$ and a set $E\subset\RR^d_{+,k}$ the restriction $a\ind_E$ (if non-trivial) is still a local $(\eta,B)$-atom.

The local version of Theorem \ref{thm:1} is the following.

%%%%%%%%%%%%%%%%%%%%%%%%%%%%%%%%%%
\begin{theorem}\label{thm:1(loc)}
Let $f\in L^1(\R_{+,k})$ and $\eta\in\Z_2^k$. The following conditions are equivalent:
\begin{enumerate}
\item $f\in h^1_\eta(\R_{+,k})$;
\item $f=\sum_{i\in\N} \lambda_i a_i+\mu_i b_i$, where $a_i/b_i$ are local $(\eta,A/B)$-atoms, and $\{\lambda_i\},\{\mu_i\}\in\ell^1(\N)$;
\item there exists $F\in h^1(\RR^d)$, supported in $\RR^{\eta,0}_+$ and such that $F\big|_{\R_{+,k}}=f$;
\item there exists $F\in h^1(\RR^d)$, supported in $(\RR^{\eta,1}_+\setminus \R_{+,k})^c$ and  such that $F\big|_{\R_{+,k}}=f$.
\end{enumerate}
Moreover, the relevant quantities are comparable, namely
\begin{equation*}%\label{eq:4(loc)}
\Vert  f\Vert_{h^1_\eta(\R_{+,k})} \simeq \inf \sum_{i\in\N}  |\lambda_i|+|\mu_i| \simeq \inf_F \Vert F\Vert_{h^1(\R)},
\end{equation*}
where the first infimum is taken over all decompositions as in (2) and the second over all $F$ as in (3) or (4).
\end{theorem}
%%%%%%%%%%%%%%%%%%%%%%%%%%%%%%%%%%%%
\begin{proof}
The proof is very similar to the proof of Theorem \ref{thm:1}, that is why we only comment on the main differences; already at this point we note that the implication \textbf{(3)$\implies$(4)} is trivial.  We fix $f\in L^1(\R_{+,k})$ throughout.

\textbf{(2)$\implies$(1)}
It suffices to check that $\eta$-extensions of local atoms supported in cubes contained in $\R_{+,k}$ are in $h^1_\eta(\R_{+,k})$ with uniform norms. If $a$ 
is a local $(\eta,B)$-atom supported in a cube $Q\subset \R_{+,k}$ such that $l(Q)> 1$, then $\calE^\eta a$ is simply a sum of $2^k$ local $(\eta,B)$-atoms.

\textbf{(1)$\implies$(3)} For this implication we use a version of Lemma \ref{lm:1} (1) for $h^1(\R)$, see Remark \ref{rm:1}, and copy the reasoning from the 
proof of the analogous implication in  Theorem \ref{thm:1}.

\textbf{(3)$\implies$(2)} With a local version of Lemma \ref{lm:2}, see Remark \ref{rm:2}, we again  copy the reasoning from the 
proof of the analogous implication in  Theorem \ref{thm:1} using, this time, \cite[Theorem 3.10]{CKS}.

\textbf{(4)$\implies$(3)} We use a version of Lemma \ref{lm:1} (2) for $h^1(\R)$, see Remark \ref{rm:1}, and copy the reasoning from the 
proof of the analogous implication in  Theorem \ref{thm:1}.
\end{proof}

As a direct consequence of equivalences  $(1) \Longleftrightarrow (3)$ in Theorems \ref{thm:1} and \ref{thm:1(loc)} we obtain the following.
%%%%%%%%%%%%%%%%%%%%%%%%%%%%%%%%% 
\begin{corollary} \label{rem:com2}
For the Weyl chamber $\R_{+,k}$ corresponding to the root system $R_k$ in $\R$, $1\le k\le d$, and for $\eta=\textbf{0}$ or $\eta=\textbf{1}$ we have 
$$
H^1_{\textbf{0}}(\R_{+,k})=H^1_{z}(\R_{+,k})\quad {\rm and}\quad H^1_{\textbf{1}}(\R_{+,k})=H^1_{r}(\R_{+,k}),
$$
with equivalence of norms. Consequently, $H^1_{\textbf{0}}(\R_{+,k})\hookrightarrow H^1_{\textbf{1}}(\R_{+,k})$ follows with  strict inclusion. 
Analogous claims hold for the local Hardy spaces. 
\end{corollary}
%%%%%%%%%%%%%%%%%%%%%%%%%%%%%%%%%%%%%%%%
In the case of $k=1$, i.e. for the root system $R_1=\{-e_d,e_d\}$ and the half-space $\R_+$, the above result for global Hardy spaces was well known 
(see, e.g. \cite[Corollaries 1.6 and 1.8]{CKS} taken for $p=1$).  

It is also worth mentioning that alternatively, the proof can be conducted by using results of \cite{AR}-arXiv. Indeed, \cite[Theorem 1 (b2)]{AR}-arXiv applied to $\Omega=\R_{+,k}$ we have $H^1_{z,a}(\R_{+,k})=H^1_{z}(\R_{+,k})$ (see \cite{AR}-arXiv for definition of $H^1_{z,a}(\R_{+,k})$). 
Hence, by the equivalence  $(1) \Longleftrightarrow (2)$ in Theorem \ref{thm:1} applied for $\eta=\textbf{0}$, $H^1_{z,a}(\R_{+,k})=H^1_{\textbf{0}}(\R_{+,k})$ 
and the first of the two claimed equalities follows. Similarly, by \cite[Theorem 1 (a2)]{AR}-arXiv, one obtains $H^1_{r,a}(\R_{+,k})=H^1_{r}(\R_{+,k})$ 
(see \cite{AR}-arXiv for definition of $H^1_{r,a}(\R_{+,k})$). Hence, by the equivalence  $(1) \Longleftrightarrow (2)$ in Theorem \ref{thm:1} applied 
for $\eta=\textbf{1}$, $H^1_{r,a}(\R_{+,k})=H^1_{\textbf{1}}(\R_{+,k})$ and the second of the two claimed equalities follows.
For the local Hardy spaces we have 
$$
h^1_{\textbf{0}}(\R_{+,k})=h^1_{z}(\R_{+,k})\quad {\rm and}\quad h^1_{\textbf{1}}(\R_+)=h^1_{r}(\R_{+,k}),
$$
with equivalence of norms, and the proof of these equalities goes analogously. This time we use \cite[Theorem 18 (b2)]{AR}-arXiv and our Theorem 
\ref{thm:1(loc)} to achieve the first equality, and \cite[Theorem 18 (a)]{AR}-arXiv and Theorem \ref{thm:1(loc)} to obtain the second one.

The proposition concluding this section  establishes relations between the $\eta$-Hardy spaces for different $\eta$'s (cf. also Corollary \ref{rem:com2}).
%%%%%%%%%%%%%%%%%%%%%%%%%%%%%%%%%%%%%%%
\begin{proposition} \label{prop:rel}
If $\eta^{(1)}\le\eta^{(2)}$ (the lexicographical order), then 
$H^1_{\eta^{(1)}}(\R_{+,k})\hookrightarrow H^1_{\eta^{(2)}}(\R_{+,k})$. In particular, we have 
$$
H^1_{\textbf{0}}(\R_{+,k})\hookrightarrow H^1_\eta(\R_{+,k})\hookrightarrow H^1_{\textbf{1}}(\R_{+,k}),
$$
with  the relevant inclusions proper when $\eta\neq\textbf{0}$ or $\eta\neq\textbf{1}$, respectively.
Analogous embeddings and statements hold for the local spaces $h^1_{\eta}(\R_{+,k})$. 
\end{proposition}
%%%%%%%%%%%%%%%%%%%%%%%%%%%%%%%%%%%%%%%
\begin{proof}
For $\eta^{(1)}\le\eta^{(2)}$ we have $\RR^{\eta^{(1)},0}_+\subset \RR^{\eta^{(2)},0}_+$. Hence the claimed continuous embedding is an immediate 
consequence of the equivalence $(1)\Leftrightarrow(3)$ in Theorem \ref{thm:1} and \eqref{eq:4}. For the local case we repeat the same argument using 
Theorem \ref{thm:1(loc)}. 

When it comes to proper inclusions this is obvious for the pair \textbf{0} and $\eta\neq\textbf{0}$ (then $H^1_\eta(\R_{+,k})$ contains functions with 
nonvanishing integral over $\R_{+,k}$). Assume \textit{a contrario}, that the spaces coincide for the pair $\eta\neq\textbf{1}$ and \textbf{1}. Then, by 
the Inverse Mapping Theorem, the corresponding norms are equivalent and this implies isomorphism (in the category of Banach spaces) of the dual spaces. 
But this is not the case, see Proposition \ref{prop:diff} and Theorem \ref{thm:3}. In the local case we repeat the latter argument for both pairs, 
\textbf{0} and $\eta\neq\textbf{0}$, and $\eta\neq\textbf{1}$ and \textbf{1}, and apply Proposition \ref{prop:diff} and Theorem \ref{thm:3(loc)}.
\end{proof}

In our final remark we point out that distinguishing $\eta$-Hardy spaces (local or global, respectively) on the base of the differences of the families of 
$(\eta,A/B)$-atoms (local or global) fails. 
To see this, given $\eta\in\Z^k_2$ we write $\mathcal A_\eta$ and $\mathcal B_\eta$ for the families of (global) $(\eta,A)$- and $(\eta,B)$-atoms, 
respectively (recall that $\mathcal B_{\textbf{0}}=\emptyset$). Similar notations obeys in the local case, we write $\mathcal A_\eta^{\rm loc}$ and 
$\mathcal B_\eta^{\rm loc}$. Although for $\eta\neq\eta'$ we have $\mathcal A_\eta\neq\mathcal A_{\eta'}$ and $\mathcal B_\eta\neq\mathcal B_{\eta'}$, 
and analogously in the local case, this is not sufficient for claiming that the $\eta$-Hardy spaces differ for different $\eta$'s. Indeed, for every 
$\eta,\eta'$ we have $\mathcal A_\eta\subset H^1_{\eta'}(\R_{+,k})$, and, when $\eta'\neq\textbf{0}$, $\mathcal B_\eta\subset H^1_{\eta'}(\R_{+,k})$.
This is because for any (classical) atom $a$ supported in $\R_{+,k}$, $\mathcal E^{\eta'}a$ is a scaled atom for any $\eta'$, and 
$\mathcal E^{\eta'}a$ is a scaled atom when $\eta'\neq0$ and $a\in L^2$ is supported in a cube $Q\subset \R_{+,k}$. In the local case we have 
$\mathcal A_\eta^{\rm loc}\cup \mathcal B_\eta^{\rm loc}\subset h^1_{\eta'}(\R_{+,k})$ for every $\eta,\eta'$ (in case $\eta'=\textbf{0}$, 
for $a\in \mathcal B_\eta$, $\mathcal E^{\eta'}a$ is a scaled local $(\eta',A)$-atom and hence in $h^1_{\eta'}(\R_{+,k})$).

%%%%%%%%%%%%%%%%%%%%%%%%%%%%%%%%%%%%%%%%%%%%%%%%%%%%
\section{$\eta$-$\BMO$ spaces on  Weyl chambers and duality}\label{sec:BMO} 
%%%%%%%%%%%%%%%%%%%%%%%%%%%%%%%%%%%%%%%%%%%%%%%%%%%%

In the first part of this section, similarly as in Section \ref{sec:semi}, let $R$, $W$, $R_+$, $C_+$, and $\Hom$ be fixed. 
Recall that $F\in L^1_{\rm loc}(\R)$ belongs to ${\rm bmo}(\R)$ if
$$
\|F\|_{{\rm bmo}(\R)}:=\sup_{l(Q)<1}\frac1{|Q|}\int_Q|F-F_Q|+\sup_{l(Q)\ge1}\frac1{|Q|}\int_Q|F|<\infty,
$$
where $F_Q$ stands for the mean value of $F$ over $Q$. Analogously, $F\in L^1_{\rm loc}(\R)$ is in $\BMO(\R)$ provided 
$$
\|F\|_{{\BMO}(\R)}:=\sup_{Q}\frac1{|Q|}\int_Q|F-F_Q|<\infty;
$$
here we identify functions modulo constants. It is well known that the norm $\|\cdot \|_{{\BMO}(\R)}$ is equivalent with $\|\cdot \|_{*, \BMO(\R)}$, where
$$
\|F\|_{*, \BMO(\R)}:=\sup_{Q}\inf_{c\in \C}\frac1{|Q|}\int_Q|F-c|.
$$
More precisely, $\|\cdot\|_{*, \BMO(\R)}\le \|\cdot\|_{\BMO(\R)}\le 2\|\cdot\|_{*, \BMO(\R)}$. Similarly, for $F\in {\rm bmo}(\R)$,
$$
\|F\|_{*, {\rm bmo}(\R)}\le \|F\|_{{\rm bmo}(\R)}\le 2\|F\|_{*, {\rm bmo}(\R)},
$$
where 
$$
\|F\|_{*, {\rm bmo}(\R)}:=\sup_{l(Q)<1}\inf_{c\in \C}\frac1{|Q|}\int_Q|F-c|+\sup_{l(Q)\ge1}\frac1{|Q|}\int_Q|F|.
$$

Following the procedure described in Section \ref{sub:pro} we define 
$$
{\rm bmo}_\eta(C_+)=\{f\in L^1_{\rm loc}(C_+)\colon \EEeta f\in {\rm bmo}(\R)\},
$$
with norm $\|f\|_{{\rm bmo}_\eta(C_+)}=\|\EEeta f\|_{{\rm bmo}(\R)}$. Analogously, 
$$
\BMO_\eta(C_+)=\{f\in L^1_{\rm loc}(C_+)\colon \EEeta f\in \BMO(\R)\},
$$
with norm $\|f\|_{\BMO_\eta(C_+)}=\|\EEeta f\|_{\BMO(\R)}$, where for  $\eta={\rm triv}$ we identify functions on $C_+$ modulo constants, and for 
$\eta\neq {\rm triv}$ elements of $\BMO_\eta(C_+)$ are genuine functions, not equivalence classes modulo constants. 

The above dichotomy is also reflected in the definition of the extension by zero of global $\BMO$ functions on open subdomains of $\R$. Namely, 
specified to $C_+$ as an open subdomain, $\BMO_z(C_+)$ is identified with the space of these (genuine) functions on $C_+$ such that $f$ extended by zero 
outside $C_+$, and denoted by the same symbol, is in $\BMO(\R)$, with inherited norm  $\|f\|_{\BMO_z(C_+)}=\|f\|_{\BMO(\R)}$; see e.g. \cite[Section 2.3]{AR}-arXiv. Unlikely, the concept of definition of $\BMO_r(C_+)$ is unchanged: $\BMO_r(C_+)$ (modulo constants) consists of restrictions to $C_+$ of $\BMO$ 
functions on $\R$ with norm $\|f\|_{\BMO(C_+)}:=\inf_{F}\|F\|_{\BMO(\R)}$, where the infimum is taken over all $F$ such that $F|_{C_+}=f$. Obviously, 
$\BMO_z(C_+)\subset \BMO_r(C_+)$ with continuous embedding in the sense that $[F]\in \BMO_r(\R)$ for $F\in\BMO_z(\R)$ and 
$\|[F]\|_{\BMO(\R)}\lesssim \|F\|_{\BMO(\R)}$. Similar convention is used in analogous situations when the space to which embedding is made consists of abstract classes. Moreover, keeping this convention in mind we shall use the symbols $\hookrightarrow$ and $\hookrightarrow_1$ in their previous sense. 

Such a dichotomy in definitions does not occur in case of local BMO spaces. We define ${\rm bmo}_z(C_+)$ and ${\rm bmo}_r(C_+)$ as the spaces of (genuine) 
functions on $C_+$ by the same principles as above, including the definitions of norms. It is then obvious that ${\rm bmo}_z(C_+)\hookrightarrow_1{\rm bmo}_r(C_+)$. 

It is worth pointing out that in particular cases of a root system $R$ and a homomorphism $\eta$ these spaces appeared in the literature. 
For example, for the half-space  $\R_+$ as a Weyl chamber corresponding to the root system $R_1=\{-e_d,e_d\}$ and $\eta=\rm triv=\textbf{0}$
or $\eta={\rm sgn}=\textbf{1}$, the corresponding spaces were denoted  in \cite{DDSY}, as $\BMO_e(\R_+)$ and $\BMO_o(\R_+)$, respectively. 
(Unfortunately, in \cite{DDSY} it was not pointed out that $\BMO_o(\R_+)$ and $\BMO_z(\R_+)$ should be treated as spaces of genuine functions, 
not equivalence classes modulo constants.)

When it comes to a comparison of local and global BMO spaces then, clearly, ${\rm bmo}\,(\R)\hookrightarrow \BMO\,(\R)$ (we have $\hookrightarrow_1$ 
for the norms $\|\cdot\|_{*,{\rm bmo}}$ and $\|\cdot\|_{*,\BMO}$). Consequently, ${\rm bmo}_z(C_+)\hookrightarrow\BMO_z(C_+)$, 
${\rm bmo}_r(C_+)\hookrightarrow\BMO_r(C_+)$, and ${\rm bmo}_\eta(C_+)\hookrightarrow\BMO_\eta(C_+)$. 

The spaces ${\rm bmo}_\eta(C_+)$, and thus also $\BMO_\eta(C_+)$, are nontrivial in the sense that they contain unbounded functions. To see this, 
recall that $F(x)=\log\frac1{|x|}\in \BMO(\R)$. Moreover, also $\Psi(x)=\max\{\log\frac1{|x|},0\}\in \BMO(\R)$. Since in addition $\Psi\in L^1(\R)$, 
it follows that $\Psi\in {\rm bmo}(\R)$. Now, choose $x_0\in C_+$ such that ${\rm dist}(x_0,\partial C_+)>1$. Then $\psi(x)=\ind_{C_+}(x)\Psi(x-x_0)$, 
$x\in C_+$, is a function on $C_+$ such that $\calE^\eta\psi\in {\rm bmo}(\R)$ ($\calE^\eta\psi$ should be seen as a sum of ${\rm bmo}(\R)$ functions, 
which are reflections of $\psi$ along a number of hyperplanes) and hence $\psi\in {\rm bmo}_\eta(C_+)$. To illustrate this general discussion 
we present in dimension one an example of an odd unbounded function that belongs to ${\rm bmo}(\RR)$ and thus also to  $\BMO(\RR)$ (surprisingly, 
we could not find such an example in the literature):
$$
\Phi(x)={\rm sgn} \,x\cdot \ind_{[1,3]}(|x|)\cdot \log\frac1{\big||x|-2\big|}, \qquad x\in\RR.
$$

We now come to the discussion of the question if the $\eta$-$\BMO$ spaces are different for different $\eta$'s. In dimension one it is well known that for 
$F$ as above, $F_o(x):={\sgn\, x}\cdot F(x)\notin \BMO(\RR)$, hence $\BMO_o(\RR_+)$ and $\BMO_e(\RR_+)$ differ. Replacing $F_o$ by $\ind_{[-1,1]}F_o$ gives 
the same conclusion in the local case, ${\rm bmo}_o(\RR_+)\neq{\rm bmo}_e(\RR_+)$ (note, however, that not only $F_o\notin{\rm bmo}(\RR)$ but also 
$F_o\notin\BMO(\RR)$). The same idea works in higher dimensions. To be precise, for $\eta'\neq\eta=\textbf{0}$, $\BMO_{\eta'}(C_+)$ and $\BMO_{\textbf{0}}(C_+)$ 
will differ in the sense that for some $f$, $[f]\in \BMO_{\textbf{0}}(C_+)$ but $f\notin \BMO_{\eta'}(C_+)$.

%%%%%%%%%%%%%%%%%%%%%%%%%%%%%%%%%%%%%%%
\begin{proposition} \label{prop:diff}
The spaces $\BMO_\eta(C_+)$, $\Hom$, are pairwise different. The same claim holds for the spaces ${\rm bmo}_\eta(C_+)$.
\end{proposition}
%%%%%%%%%%%%%%%%%%%%%%%%%%%%%%%%%%%%%%%
\begin{proof}
We can assume that $d\ge2$. Let $\eta^{(1)}\neq \eta^{(2)}$, say, $\eta^{(1)}(\sigma_\a)=1$ and  $\eta^{(2)}(\sigma_\a)=-1$ for some $\a\in\Sigma$. 
Take $x_0$, `well inside' the $(d-1)$-dimensional interior of the facet $\mathcal F_\alpha=\overline{C_+}\cap \langle\a\rangle^\bot$, i.e. in the 
distance not smaller than 2 from the origin and from any of the $(d-2)$-dimensional faces (if any) of $C_+$, i.e. intersections of $\mathcal F_\a$ 
with any other different wall of $C_+$. Consider $\hat{\psi}$, the restriction to $C_+$ of $\psi(x)=\Psi(x-x_0)$, $x\in\R$, the latter function being  
in ${\rm bmo}(\R)$ ($\Psi$ as earlier defined). Then $\calE^{\eta^{(1)}}\hat{\psi}\in{\rm bmo}(\R)$. This is because $\calE^{\eta^{(1)}}\hat{\psi}$ is 
a linear combination, with $\pm1$ coefficients, of shifts of $\Psi$, with $\psi$ as one of the summands; here radiality of $\Psi$ plays a role. 
On the other hand, $\calE^{\eta^{(2)}}\hat{\psi}\notin\BMO(\R)$. This is because $\calE^{\eta^{(2)}}\hat{\psi}$ is a linear combination (with $\pm1$ 
coefficients) of shifts of the rotated function 
$$
\tilde\psi(x)=\psi(x)\ind_{\langle\a\rangle^\bot_+}-\psi(x)\ind_{\langle\a\rangle^\bot_-}, \qquad x\in\R,
$$
i.e. the function $\psi$ broken along the hyperplane $\langle\a\rangle^\bot$; here $\langle\a\rangle^\bot_{\pm}$ stand for the `upper' and `lower' 
half-spaces with $\langle\a\rangle^\bot$ as the supporting hyperplane ($C_+$ is contained in $\langle\a\rangle^\bot_{+}$).  The oscillations 
of $\calE^{\eta^{(2)}}\hat{\psi}$, which agrees on $B(x_0,1)$ with $\tilde\psi$, over small cubes centered at $x_0$ are unbounded. 

This proves the two claims of the proposition.
\end{proof}

\subsection{The case of orthogonal root systems} \label{sub:ort} 
We now leave the general case and specify considerations to orthogonal root systems. Similarly as in the case of Hardy spaces, we shall characterize 
$\BMO_\eta(\R_{+,k})$ and ${\rm bmo}_\eta(\R_{+,k})$ for the orthogonal root systems $R_k$; this is the contents of Theorems \ref{thm:2} and \ref{thm:2(loc)}.
Hence, in this and the next subsection, we assume $1\leq k\leq d$ to be fixed. Also, we make use of the notation introduced in the previous section. 

As mentioned before, we work simultaneously with function spaces and with function spaces modulo constants. In order to keep the notation transparent, 
we simply write $F\in\BMO(\R)$ even if $F$ is a function (and not an abstract class modulo constants). This is a short form of saying that 
$F\in L^1_{\mathrm{loc}}(\R)$ and $[F]\in \BMO(\R)$. The same comment applies to $\BMO_{\bf 0}(\R_{+,k})$ and $\BMO_z(\R_{+,k})$.

For the proof of the next theorem we shall need an auxiliary lemma which, in some sense, is a dual result to Lemma \ref{lm:1}. Notice that 
Lemma \ref{lm:3} (1) is true for abstract classes modulo constants as well. 

\begin{lemma}\label{lm:3}
Let $e$ be a coordinate vector and $F\in\BMO(\RR^d)$. Then:
	\begin{enumerate}
		\item the even (with respect to $\langle e\rangle^\bot$) extension of $F\big|_{\langle e\rangle^\bot_+}$ to $\RR^d$, namely 
		$F\ind_{\langle e\rangle^\bot_+}+(F\ind_{\langle e\rangle^\bot_+})\circ \sigma_e$, is in $\BMO(\RR^d)$, and 
		$$
		\big\Vert F\ind_{\langle e\rangle^\bot_+}+(F\ind_{\langle e\rangle^\bot_+})\circ \sigma_e\big\Vert_{*,\BMO(\RR^d)}\leq 2\Vert F\Vert_{*,\BMO(\RR^d)};
		$$
		\item if $F$ is odd with respect to $\langle e\rangle^\bot$, then $F\ind_{\langle e\rangle^\bot_+}\in \BMO(\RR^d)$ and 
		$$
		\Vert F\ind_{\langle e\rangle^\bot_+}\Vert_{\BMO(\RR^d)}\lesssim  \Vert F\Vert_{\BMO(\RR^d)}.
		$$
	\end{enumerate}
\end{lemma}

\begin{proof}
	\noindent{\textbf{(1)}} Fix $F\in\BMO(\R)$ and let $Q$ be a cube in $\RR^d$. Our aim is to prove
	$$
\inf_{c\in\C} \frac1{|Q|} \int_{Q} \big| F\ind_{\langle e\rangle^\bot_+}+(F\ind_{\langle e\rangle^\bot_+})\circ\sigma_e - c\big|\leq 2\Vert F\Vert_{*,BMO(\RR^d)}. 
	$$
The only non-trivial case is when $Q\cap \langle e\rangle^\bot\neq \emptyset$. We denote $Q_{\pm}=Q\cap \langle e\rangle^\bot_{\pm}$. For any $c\in\C$ we obtain
	\begin{equation*}
\int_{Q}\big|F\ind_{\langle e\rangle^\bot_+}+(F\ind_{\langle e\rangle^\bot_+})\circ\sigma_e-c\big|=\int_{Q_+}\big|F-c\big| +\int_{\sigma_e(Q_{-})}\big|F-c\big|.
	\end{equation*}
Let $\widehat{Q}\subset \langle e\rangle^\bot_+$ be the cube such that $Q_+\cup\sigma_e(Q_{-})\subset \widehat{Q}$ and $|\widehat{Q}|= |Q|$. Thus,
	\begin{equation*}
	\inf_{c\in\C}\frac1{|Q|}\int_{Q}\big|F\ind_{\langle e\rangle^\bot_+}+(F\ind_{\langle e\rangle^\bot_+})\circ \sigma_e - c\big|\leq \inf_{c\in\C} \frac2{|\widehat{Q}|}  \int_{\widehat{Q}} \big| F - c\big|\leq 2\Vert F\Vert_{*,\BMO(\RR^d)}.
	\end{equation*}
	
	\noindent{\textbf{(2)}} Fix  $F\in \BMO(\RR^d)$ such that $F\circ\sigma_e=-F$. Let $Q$ be a cube in $\RR^d$. We shall check that 
	$$
	\inf_{c\in\C} \frac1{|Q|} \int_{Q} \big| F\ind_{\langle e\rangle^\bot_+}-c\big| \lesssim \Vert F\Vert_{\BMO(\RR^d)}. 
	$$
As before, we assume $Q\cap \langle e\rangle^\bot\neq\emptyset$ and we use the same notation. Observe that
\begin{align*}
	\inf_{c\in\C} \frac1{|Q|}\int_{Q}\big| F\ind_{\langle e\rangle^\bot_+}-c\big| \leq  \inf_{c\in\C} \frac1{|Q|}\int_{Q}| F-c|+  \frac1{|Q|}\int_{Q_{-}}| F|\leq \Vert F\Vert_{*,\BMO(\RR^d)} + \frac1{2|Q|}\int_{Q_{-}\cup \sigma_e(Q_{-})}| F|.
\end{align*}
Let $\widehat{Q}\subset \RR^d$ be a cube symmetric in $\langle e\rangle^\bot$, such that $Q_{-}\cup \sigma_e(Q_{-})\subset \widehat{Q}$ and $|\widehat{Q}| < 2^d|Q| $. Hence,
	\begin{align*}
	\frac1{2|Q|}\int_{Q_{-}\cup \sigma_e(Q_{-})}| F|\leq  \frac{2^{d-1}}{|\widehat{Q}|}\int_{\widehat{Q}}| F| \leq 2^{d-1} \Vert F\Vert_{\BMO(\RR^d)},
	\end{align*}
where in the second inequality we used the fact that $F_{\widehat{Q}}=0$. This concludes the proof.	
\end{proof}

We shall need the fact (which is certainly a folklore result), that the definition of the ${\rm bmo}(\R)$ norm does not depend on choosing 1 as a breaking 
point, i.e. 1 can be replaced by an arbitrary positive number with equivalence of norms. 
For $a>0$ and $F\in L^1_{\rm loc}(\R)$ we introduce the quantity
$$
\|F\|_{*, {\rm bmo}_a(\R)}:=\sup_{l(Q)<a}\inf_{c\in \C}\frac1{|Q|}\int_Q|F-c|+\sup_{l(Q)\ge a}\frac1{|Q|}\int_Q|F|
$$
and define ${\rm bmo}_a(\R)=\{F\colon \|F\|_{*, {\rm bmo}(\R);a}<\infty\}$ so that ${\rm bmo}_1(\R)$ is just ${\rm bmo}(\R)$. 

The  proof of the the following technical result is simple but we  provide details for completeness.
%%%%%%%%%%%%%%%%%%%%%%%%%%%%%%%%%%%%%%%%%%%
\begin{proposition} \label{pro:app2}
Let $0<a,b<\infty$. Then the spaces ${\rm bmo}_{*,a}(\R)$ and ${\rm bmo}_{*,b}(\R)$ coincide with equivalence of norms (and implicit constants depending on 
$a$ and $b$).
\end{proposition}
%%%%%%%%%%%%%%%%%%%%%%%%%%%%%%%%%%%%%%%%%%%
\begin{proof}
We explain the idea of the proof for $a=1$ and $b=2$; generalization to arbitrary $0<a<b<\infty$ is straightforward. For notational simplicity 
we let $|F|_Q=\frac1{|Q|}\int_Q|F|$ and $|F-c|_Q=\frac1{|Q|}\int_Q|F-c|$. 

It is clear that 
\begin{align*}
\|F\|_{*, {\rm bmo}_2(\R)}&=\sup_{l(Q)<2}\inf_{c\in \C}|F-c|_Q+\sup_{l(Q)\ge 2}|F|_Q\\
&\le \sup_{l(Q)<1}\inf_{c\in \C}|F-c|_Q+\sup_{l(Q)\ge 1}|F|_Q+\sup_{1\le l(Q)<2}\inf_{c\in \C}|F-c|_Q.
\end{align*}
Obviously,
$$
\sup_{1\le l(Q)<2}\inf_{c\in \C}|F-c|_Q\le  \sup_{1\le l(Q)<2}|F|_Q\le \sup_{l(Q)\ge 1}|F|_Q,
$$
so that $\|F\|_{*, {\rm bmo}_2(\R)}\le 2\|F\|_{*, {\rm bmo}_1(\R)}$ follows. 

For the opposite inequality we have  
\begin{align*}
\|F\|_{*, {\rm bmo}_1(\R)}&=\sup_{l(Q)<1}\inf_{c\in \C}|F-c|_Q+\sup_{l(Q)\ge 1}|F|_Q\\
&\le \sup_{l(Q)<2}\inf_{c\in \C}|F-c|_Q+\sup_{l(Q)\ge 2}|F|_Q+\sup_{1\le l(Q)<2}|F|_Q.
\end{align*}
To estimate the latter summand take a cube $\widehat Q$ with $1\le l(\widehat Q)<2$. It follows that $|F|_{\widehat Q}\le 2^d |F|_{2\widehat Q}$ and hence,  
$$
\sup_{1\le l(Q)<2}|F|_Q\le 2^d\sup_{l(Q)\ge2}|F|_Q.
$$
This leads to $\|F\|_{*, {\rm bmo}_1(\R)}\le (2^d+1)\|F\|_{*, {\rm bmo}_2(\R)}$.
\end{proof}

\begin{remark} \label{rm:4}
Lemma \ref{lm:3} remains valid in the `local' setting, i.e. when in the statement we replace $\BMO(\RR^d)$ by ${\rm bmo}(\RR^d)$. The proof follows that of 
Lemma \ref{lm:3} with the following changes. For the proof of the relevant statement replacing (1), for fixed $F\in{\rm bmo}(\R)$  
and with notation $\tilde F=F\ind_{\langle e\rangle^\bot_+}+(F\ind_{\langle e\rangle^\bot_+})\circ\sigma_e$, we wish to check that 
$$
\sup_{l(Q)<1}\inf_{c\in\C} \frac1{|Q|}\int_{Q} \big|\tilde F-c\big|+\sup_{l(Q)\ge1}\frac1{|Q|}\int_{Q}\big|\tilde F\big|\le 2\Vert F\Vert_{*,{\rm bmo}(\RR^d)}. 
$$
To achieve this we take a cube $Q$ in $\RR^d$ and consider separately the cases $l(Q)<1$ and $l(Q)\ge1$ following the argument from the proof of (1) 
in Lemma \ref{lm:3}. 

For the proof of the relevant statement replacing (2), for fixed $F\in{\rm bmo}(\R)$ we wish to check that 
$$
\sup_{l(Q)<1}\inf_{c\in\C} \frac1{|Q|} \int_{Q} \big|F\ind_{\langle e\rangle^\bot_+}  - c\big|+
\sup_{l(Q)\ge1} \frac1{|Q|} \int_{Q} \big|F\ind_{\langle e\rangle^\bot_+} \big| \leq 2\Vert F\Vert_{*,{\rm bmo}(\RR^d)}. 
$$
Again, given a cube $Q$ in $\RR^d$ we consider separately the cases $l(Q)<1$ and $l(Q)\ge1$ following the argument from the proof of (2) in Lemma \ref{lm:3}. 
This time, considering the larger cube $\widehat Q$, we apply Proposition \ref{pro:app2}. With this, one can easily reach the desired claim. 
\end{remark}

For the proofs of equivalences $(1)\Longleftrightarrow (2)$ in Theorems \ref{thm:2} and \ref{thm:2(loc)} it is reasonable to introduce convenient notation. 
Given a suitable function $f$ on $\R_{+,k}$ and a parameter $0<a\le\infty$ we define
$$
m_{{\rm osc},a}(f)=\sup_{l(Q')<a}\inf_{c\in \C}|f-c|_{Q'} \quad {\rm and} \quad m_{{\rm mv},a}(f)=\max\{m^{(1)}_{{\rm mv},a}(f), m^{(2)}_{{\rm mv},a}(f)\},
$$
where
$$
m^{(1)}_{{\rm mv},a}(f)=\sup_{l(Q')\ge a}|f|_{Q'} \quad {\rm and} \quad  m^{(2)}_{{\rm mv},a}(f)=\sup_{\substack{l(Q'')<a\\Q'' adjacent}}|f|_{Q''}. 
$$
Here $Q'$ and $Q''$ run over the collections of cubes in $\R_{+,k}$ satisfying the corresponding length restrictions, and in case of $m^{(2)}_{{\rm mv},a}(f)$,
additionally $Q''$ are  \textit{adjacent} (in the sense as before the statement of Theorem \ref{thm:2}); $|f-c|_{Q'}$ and $|f|_{Q''}$ denote the  the 
\textit{mean values} of $|f-c|$ and $|f|$ over corresponding cubes. Notice that only $m^{(2)}_{{\rm mv},a}(f)$ depends on $\eta$ (which is not indicated) 
and for $\eta=\textbf{0}$ (there are no adjacent cubes then) we set $m^{(2)}_{{\rm mv},a}(f)=0$. For $a=1$ we shall drop $a$ in our notation writing simply 
$m_{\rm osc}(f)$, etc. 

In the theorem that follows mind that for $\eta={\bf 0}$ the statements concern abstract classes of functions, i.e. both $f$ and $F$ are considered modulo constants (it is clear what then $F\big\vert_{\R_{+,k}}=f$ means); otherwise, for $\eta\neq{\bf 0}$, the statements concern genuine functions. Also, notice 
that for $\eta={\bf 0}$ the assumptions imposed on the support of $F$ are void (note that $\RR^{{\bf 0},1}_+=\R$ and $\RR^{{\bf 0},0}_+=\R_{+,k}$) and (3) 
and (4) are identical. In addition, observe that since $\RR^{\eta,1}_+\subset \big(\RR^{\eta,0}_+\setminus \R_{+,k}\big)^{\mathrm{c}}$, for $\eta\neq{\bf 0}$ condition (4) is seemingly weaker than condition (3), but in fact these conditions occur to be equivalent. It is also worth pointing out a duality between conditions (3) and (4) in Theorem \ref{thm:1}, and conditions (3) and (4) below. Finally, we define
\begin{equation*}
M_1(f)=\sup_{Q'}\frac1{|Q'|}\int_{Q'}|f-f_{Q'}|, \qquad M_2(f)=\sup_{Q''}\frac1{|Q''|}\int_{Q''}|f|,
\end{equation*}
where the first supremum is taken over all cubes $Q'\subset\R_{+,k}$ and the second supremum is taken over all cubes  $Q''\subset\R_{+,k}$ 
 adjacent to at least one of the walls of $\R_{+,k}$ with the minus sign attached by $\eta$;

%%%%%%%%%%%%%%%%%%%%%%%%%%%%%%%%  
\begin{theorem}\label{thm:2}
Let $f\in L^1_{\mathrm{loc}}(\R_{+,k})$ and $\eta\in\Z_2^k$. The following conditions are equivalent:
	\begin{enumerate}
		\item $f\in \BMO_\eta(\R_{+,k})$;
		\item $M_1(f)+M_2(f)<\infty$; 
		\item there exists $F\in \BMO(\RR^d)$, supported in $\RR^{\eta,1}_+$ and such that $F\big|_{\R_{+,k}}=f$;
		\item there exists $F\in \BMO(\RR^d)$, supported in  $(\RR^{\eta,0}_+\setminus \R_{+,k})^{\mathrm{c}}$ and such that $F\big|_{\R_{+,k}}=f$.
	\end{enumerate}
Moreover, $\Vert f\Vert_{\BMO_\eta(\R_{+,k})}\simeq M_1(f)+M_2(f)\simeq \inf_{F} \Vert F\Vert_{\BMO(\R)}$, where the infimum is taken over all $F$ as in (3)  
or in (4), respectively.
\end{theorem}
%%%%%%%%%%%%%%%%%%%%%%%%%%%%%%%%

\begin{proof} 
 \textbf{(1)$\Longleftrightarrow$(2)} 
Let $F=\calE^\eta f$ for $f\in L^1_{\rm loc}(\R_{+,k})$. Proving the desired equivalence together with equivalence of relevant quantities can be replaced by verifying that
\begin{equation}\label{abc'}
\|F\|_{*,\BMO(\R)}\lesssim m_{{\rm osc},\infty}(f)+m^{(2)}_{{\rm mv},\infty}(f)\lesssim \|F\|_{*,\BMO(\R)},
\end{equation}
with implicit constants independent of $f$. 
Notice that $m_{{\rm osc},\infty}(f)$ measures the total  oscillation of $f$ over all cubes contained in $\R_{+,k}$, and $m^{(2)}_{{\rm mv},\infty}(f)$ 
gives the supremum of the mean values of $|f|$ over all adjacent cubes (if any) contained in $\R_{+,k}$. Therefore proving the RHS in \eqref{abc'} reduces 
to checking that
$$
m^{(2)}_{{\rm mv},\infty}(f)\lesssim \|F\|_{*,\BMO(\R)},
$$
and additionally we can assume that $\eta\neq \textbf{0}$. Let $Q''\subset \R_{+,k}$ be adjacent. To verify that $|f|_{Q''}\lesssim \|F\|_{*,\BMO(\R)}$  
let $\widehat{Q''}$ be a smallest cube in $\R$ containing  $Q''$ and symmetric with respect to any wall $\mathcal W_i$ of $\R_{+,k}$ with 
${\rm sign}_\eta(\mathcal W_i)=-1$. Then  $F_{\widehat{Q''}}=0$, the volumes of $Q''$ and $\widehat{Q''}$ are comparable and hence 
$$
\frac1{|Q''|}\int_{Q''}|f|\lesssim \frac1{|\widehat{Q''}|}\int_{\widehat{Q''}}|F|\le \|F\|_{\BMO(\R)}\lesssim \|F\|_{*, \BMO(\R)}.
$$

To prove the LHS in \eqref{abc'} note that  due to the symmetry reasons, estimating the mean values $|F-c|_Q$, the only cubes that matter are the ones 
crossed by at least one of the walls of $\R_{+,k}$ (for the remaining cubes these mean values are dominated by $m_{{\rm osc},\infty}(f)$). Let 
$Q=\prod_{j=1}^d(\a_j,\b_j)$ be such a cube and let $I$ denote the set of these $1\le i\le k$ such that $\mathcal W_i$ crosses $Q$, i.e. 
$\a_{n-k+i}<0<\b_{n-k+i}$. Again due to the symmetry reasons, we can assume that  for every $i\in I$, $\b_{n-k+i}\ge|\a_{n-k+i}|$. 
Assume first that for every $i\in I$ we have ${\rm sign}_\eta(\mathcal W_i)=1$ (this includes the case $\eta=\textbf{0}$).  
Let  $Q'$ be a smallest cube in  $\R_{+,k}$ that contains $Q^+=Q\cap \R_{+,k}$. Then the volumes of $Q$ and $Q'$ are comparable, and $Q'$ 
together with all the reflections of $Q'$ with respect to the walls  $\mathcal W_i$, $i\in I$, cover $Q$. Consequently we have
\begin{equation}\label{xyz}
\frac1{|Q|}\int_Q|F-c|\lesssim \frac1{|Q'|}\int_{Q'}|f-c|\le m_{{\rm osc},\infty}(f).
\end{equation}
In the complementary case, when ${\rm sign}_\eta(\mathcal W_i)=-1$ for some $i\in I$, denote $I^-=\{i\in I\colon {\rm sign}_\eta(\mathcal W_i)=-1\}$.  
Let $Q''\subset \R_{+,k}$ be a smallest cube containing $Q\cap\R_{+,k}$. Then $Q''$ is adjacent, the volumes of $Q$ and $Q''$ are comparable, and 
$Q''$ together with all its reflections with respect to  the walls  $\mathcal W_i$, $i\in I^-$, cover $Q$. Consequently,
\begin{equation}\label{xyz'}
\inf_{c\in \C}|F-c|_Q \le \frac1{|Q|}\int_Q|F|\lesssim \frac1{|Q''|}\int_{Q''}|f|\lesssim m^{(2)}_{{\rm mv},\infty}(f).
\end{equation}

\noindent	\textbf{(4)$\implies$(1)} Assuming that $F$ is such as in (4)  it suffices to check that $\calE^\eta f\in \BMO(\RR^d)$ together with a corresponding bound. It is convenient to single out the case $\eta=\textbf{1}$. Then we have to our disposal $F\in \BMO(\RR^d)$ supported in $\R_{+,k}$ and such that 
$F\big|_{\R_{+,k}}=f$. It is clear that  $\calE^{\eta}f=\sum_{g\in W(R_k)} F\circ g$ and hence $\calE^\eta f\in \BMO(\RR^d)$ with 
$\|f\|_{\BMO_\eta(\R)}=\|\calE^\eta f\|_{\BMO(\R)}\lesssim \|F\|_{\BMO(\R)}$.

Assume now that $\eta\neq\textbf{1}$. Let $1\le\eta_{i_1}<\ldots<\eta_{i_s}\le d$, $1\le s\le k$, be all the entries of $\eta$ which are equal to 0. 
We apply Lemma \ref{lm:3} (1) to $e_{i_1}$ and $F$ to obtain  $F_1\in \BMO(\RR^d)$ being the even extension of $F|_{\langle e_{i_1}\rangle_+^\bot}$. 
Then we continue this process to $e_{i_2}$ and $F_1$ to get  the even extension of $F_1|_{\langle e_{i_2}\rangle_+^\bot}$, and so on. 
The final function $F_s$ belongs to $\BMO(\RR^d)$ and $\|F_s\|_{*,\BMO(\RR^d)}\lesssim\|F\|_{*,\BMO(\RR^d)}$. 
Moreover, by the assumptions imposed  on $F$ and the way $F_s$ was constructed, we have $F_s=(\calE^\eta f)\ind_{\RR^{\eta,1}_+}$. 
In particular, $F_s$ is supported in $\RR^{\eta,1}$ and $F_s|_{\R_{+,k}}=f$. Thus, for $\eta=\textbf{0}$ we are done.
		
Notice that for $\eta={\bf 0}$ the above argument is valid modulo constants. Notably, this finishes the proof of the theorem in this case. Indeed, 
	(1) clearly implies (3) and (4), which are identical. Thus, in the next parts of the proof we focus on the case $\eta\neq {\bf 0}$.
	
To continue this part for $\textbf{0}\neq \eta\neq\textbf{1}$ we shall use an argument similar to that from the proof of 
 \textbf{(4)$\implies $(1)} in Theorem \ref{thm:1}. Namely, consider the complementary root (sub)systems of $R_k$, 
	$$
	R_{k,0}=\{\pm e_{d-k+i}\colon \eta_i=0\}\quad {\rm and} \quad R_{k,1}=\{\pm e_{d-k+i}\colon \eta_i=1\}. 
	$$
As the positive Weyl chambers corresponding to these root systems we choose $\mathbb R^{\eta,0}_+$ and $\mathbb R^{\eta,1}_+$, respectively. We have the natural 
identification of $W(R_k)$ with the direct product $W(R_{k,1})\oplus W(R_{k,0})$, and also $\eta$ can be viewed as the 'tensor product' of the corresponding homomorphisms on $W(R_{k,1})$ and $W(R_{k,0})$, namely $\eta=\eta^{(1)}\otimes\eta^{(0)}$, where $\eta^{(1)}=(1,\ldots,1)$ ($\#\{i\colon \eta_i=1\}$-times), and
$\eta^{(0)}=(0,\ldots,0)$ ($\#\{i\colon \eta_i=0\}$-times). Due to the product structure of $W(R_k)$ and the tensor product structure of $\eta$ it is easily 
seen that $\calE^{\eta}f=\calE^{\eta^{(1)}}(\calE^{\eta^{(0)}}f)$ for any $f$ supported in $\mathbb R_{+,k}$; notice that if $f$ is such a function, 
then $\calE^{\eta^{(0)}} f$ is supported  in $\mathbb R^{\eta,1}_+$ so that the composition makes sense.

To finish this part of the proof we argue as follows. The equality $F_s=(\calE^\eta f)\ind_{\RR^{\eta,1}_+}$ can be reinterpreted as $F_s=\calE^{\eta^{(0)}}f$ 
and hence $\calE^\eta f=\calE^{\eta^{(1)}}(\calE^{\eta^{(0)}}f)=\calE^{\eta^{(1)}} F_s$. But $F_s$ is supported in $\RR^{\eta,1}$ so
$$
\|f\|_{\BMO_\eta(\RR^d)}=\Vert(\calE^{\eta}f) \Vert_{\BMO(\RR^d)}=\Vert(\calE^{\eta^{(1)}}F_s) \Vert_{\BMO(\RR^d)}\lesssim\Vert F_s\Vert_{\BMO(\RR^d)}
\lesssim\Vert F\Vert_{\BMO(\RR^d)}
$$
with the implicit constant independent of $f$. 	
	
\noindent \textbf{(1)$\implies$(3)} Let $f\in \BMO_\eta(\R_{+,k})$. This means that $\calE^{\eta}f\in \BMO(\R)$ and clearly $\calE^{\eta}f$  is odd 
in each $\langle e_{d-k+i}\rangle^\bot$ such that $\eta_i=1$ (if any). Hence, applying Lemma \ref{lm:3} (2) consecutively to each such $e_{d-k+i}$ (if any) 
yields $\big(\calE^{\eta}f\big)\ind_{\RR^{\eta,1}_+}\in \BMO(\R)$. Consequently,  $F:=(\calE^\eta f)\ind_{\RR_+^{\eta,1}}$ has the required properties and 
	\begin{equation*}
	\Vert F\Vert_{\BMO(\RR^d)}\lesssim \Vert f\Vert_{\BMO_\eta(\R_{+,k})}
	\end{equation*}
follows with the implicit constant independent of $f$. 
	
\noindent	\textbf{(3)$\implies$(4)} This is obvious; clearly, the former inequality remains in force.
	
Finally, observe that the claimed equivalence of relevant quantities follows from the partial conclusions included in the proofs of the three implications.
\end{proof}

The local version of Theorem \ref{thm:2} is the following (note that for $\eta={\bf 0}$ the assumptions imposed on the support of $F$ are void and (3) and (4)  are identical). Again, observe a duality between conditions (3)  and (4) in Theorem \ref{thm:1(loc)}, and conditions (3) and (4)  below. We define
\begin{equation*}
M_1^{\rm loc}(f)=\sup_{Q'}\frac1{|Q'|}\int_{Q'}|f-f_{Q'}|, \qquad M_2^{\rm loc}(f)=\sup_{Q''}\frac1{|Q''|}\int_{Q''}|f|<\infty,
\end{equation*}
where  $Q',Q''$ are cubes included in $\R_{+,k}$ with the following restrictions: $l(Q')<1$ in the first supremum, and in the second supremum 
either $l(Q'')\ge1$ or $Q''$ is adjacent to at least one of the walls of $\R_{+,k}$ with the minus sign attached by $\eta$ and $l(Q'')<1$.
%%%%%%%%%%%%%%%%%%%%%%%%%%%%%%%%
\begin{theorem}\label{thm:2(loc)}
Let $f\in L^1_{\mathrm{loc}}(\R_{+,k})$ and $\eta\in\Z_2^k$. The following conditions are equivalent:
	\begin{enumerate}
		\item $f\in {\rm bmo}_\eta(\R_{+,k})$;
		\item $M_1^{\rm loc}(f)+M_2^{\rm loc}(f)<\infty$;
		\item there exists $F\in {\rm bmo}(\RR^d)$, supported in $\RR^{\eta,1}_+$ and such that $F\big|_{\R_{+,k}}=f$;
		\item there exists $F\in {\rm bmo}(\RR^d)$, supported in  $(\RR^{\eta,0}_+\setminus \R_{+,k})^{\mathrm{c}}$ and such that $F\big|_{\R_{+,k}}=f$.
	\end{enumerate}
Moreover, $\Vert f\Vert_{{\rm bmo}_\eta(\R_{+,k})}\simeq M_1^{\rm loc}(f)+M_2^{\rm loc}(f)\simeq \inf_{F} \Vert F\Vert_{{\rm bmo}(\R)}$, where the 
infimum is taken over all $F$ as in (3) or in (4), respectively.
\end{theorem}
%%%%%%%%%%%%%%%%%%%%%%%%%%%%%%%%
\begin{proof}
The proof is similar to that of Theorem \ref{thm:2} so we only provide some comments (skipping the implication \textbf{(3)$\implies$(4)} which is obvious). 

\textbf{(1)$\Longleftrightarrow$(2)} Similarly to the above notation, for a suitable function $F$ on $\R$ define
$$
M_{\rm osc}(F)=\sup_{l(Q)<1}\inf_{c\in \C}|F-c|_Q \quad {\rm and} \quad M_{\rm mv}(F)=\sup_{l(Q)\ge 1}|F|_Q;
$$
here $Q$ run over the collections of all cubes in $\R$ satisfying  the corresponding length restrictions. By definition, 
$\|F\|_{*, {\rm bmo}(\R)}=M_{\rm osc}(F)+M_{\rm mv}(F)$. Now, let $F=\calE^\eta f$ for $f\in L^1_{\mathrm{loc}}(\R_{+,k})$. We are reduced to verifying that
\begin{equation}\label{abc''}
M_{\rm osc}(F)+M_{\rm mv}(F)\lesssim m_{\rm osc}(f)+m_{\rm mv}(f)\lesssim M_{\rm osc}(F)+M_{\rm mv}(F).
\end{equation}
We begin with comments on the RHS of \eqref{abc''}. Obviously, $m_{\rm osc}(f)\le M_{\rm osc}(F)$ and $m^{(1)}_{\rm mv}(f)\le M_{\rm mv}(F)$, 
hence it suffices only to check that  $m^{(2)}_{\rm mv}(f)\lesssim M_{\rm osc}(F)+M_{\rm mv}(F)$, additionally assuming  that $\eta\neq \textbf{0}$.
Let $Q''\subset \R_{+,k}$ be adjacent with $l(Q'')<1$, and let $\widehat{Q''}$ be a smallest cube in $\R$ containing  $Q''$ and symmetric with 
respect to any wall $\mathcal W_i$ of $\R_{+,k}$ with ${\rm sign}_\eta(\mathcal W_i)=-1$. Then  the volumes of $Q''$ and $\widehat{Q''}$ are comparable 
and $F_{\widehat{Q''}}=0$. Hence 
$$
\frac1{|Q''|}\int_{Q''}|f|\lesssim \frac1{|\widehat{Q''}|}\int_{\widehat{Q''}}|F|\le \|F\|_{{\rm bmo}(\R)}\lesssim \|F\|_{*,{\rm bmo}(\R)}.
$$

To prove the LHS in \eqref{abc''}, estimating the mean values $|F-c|_Q$ when $l(Q)<1$ and $|F|_Q$ when $l(Q)\ge1$, again only the cubes  
crossed by at least one of the walls of $\R_{+,k}$ matter. We proceed as in the proof of the LHS in \eqref{abc'}, keeping the same notation and 
applying the appropriate reducing observations but additionally we split the reasoning to the cases  $l(Q)<1$ and $l(Q)\ge1$. 

Assume first that $l(Q)<1$. Then, in the case when for every $i\in I$ we have ${\rm sign}_\eta(\mathcal W_i)=1$, reaching \eqref{xyz} the only change 
to be done is replacement of $m_{{\rm osc},\infty}(f)$ by $m_{{\rm osc},2}(f)$ (this is since $l(Q')<2$). In the complementary case, when 
${\rm sign}_\eta(\mathcal W_i)=-1$ for some $i\in I$, reaching \eqref{xyz'} the change to be done is replacement of $m^{(2)}_{{\rm mv},\infty}(f)$ by 
$m^{(2)}_{{\rm mv}}(f)$ ($l(Q'')<l(Q)$). Concluding, we have
$$
M_{\rm osc}(F)\lesssim m_{{\rm osc},2}(f)+m^{(2)}_{{\rm mv}}(f)\lesssim m_{{\rm osc}}(f)+m_{{\rm mv}}(f)
$$
(we used the fact that $m_{{\rm osc},2}(f)\lesssim m_{{\rm osc}}(f)+m^{(1)}_{{\rm mv}}(f)$; see the proof of Proposition \ref{pro:app2}).

Assume now that $l(Q)\ge1$. Taking $c=0$ in \eqref{xyz} we can replace $m_{{\rm osc},\infty}(f)$ by $m^{(1)}_{{\rm mv}}(f)$. Similarly, 
disregarding $\inf_{c\in \C}|F-c|_Q$ and replacing $m^{(2)}_{{\rm mv},\infty}(f)$ by $m^{(1)}_{{\rm mv}}(f)$ in \eqref{xyz'}, and then 
taking the supremum over all $Q$, $l(Q)\ge1$, leads to
$$
M_{\rm mv}(F)\lesssim m^{(1)}_{{\rm mv}}(f).
$$
The proof of the LHS in \eqref{abc''} is completed.

\textbf{(4)$\implies$(1)} To check that $\calE^\eta f\in {\rm bmo}(\RR^d)$ together with a relevant bound, only cosmetic changes are needed in the case 
$\eta=\textbf{1}$. Treating the case $\eta\neq\textbf{1}$ we apply Remark \ref{rm:4} in place of Lemma \ref{lm:3} and follow the procedure described in the proof 
of \textbf{(4) $\implies$(1)}, Theorem \ref{thm:2}; this is sufficient to close the case $\eta=\textbf{0}$. Finally, the case $\textbf{0}\neq\eta\neq\textbf{1}$ 
also needs only cosmetic changes.

\textbf{(1)$\implies$(3)} The argument for this implication copies that of \textbf{(1)$\implies$(3)} in the proof of Theorem \ref{thm:2} with obvious changes.
\end{proof}

As a direct consequence of Theorems \ref{thm:2} and \ref{thm:2(loc)}, by using the equivalences \textbf{(1)$\Longleftrightarrow$(3)}, we obtain the following.
%%%%%%%%%%%%%%%%%%%%%%%%%%%%%%%%%%%
\begin{corollary} \label{rem:com_bmo}
For the Weyl chamber $\R_{+,k}$ corresponding to the root system $R_k$ in $\R$, we have
$$
\BMO_{\textbf{0}}(\R_{+,k})=\BMO_{r}(\R_{+,k})\quad {\rm and}\quad \BMO_{\textbf{1}}(\R_{+,k})=\BMO_{z}(\R_{+,k}),
$$ 
with equivalence of norms. In particular, it follows that $\BMO_{\textbf{1}}(\R_{+,k})\hookrightarrow\BMO_{\textbf{0}}(\R_{+,k})$ with  strict inclusion. 
Analogous statement holds for local ${\rm bmo}$ spaces.
\end{corollary}
%%%%%%%%%%%%%%%%%%%%%%%%%%%%%%%

\subsection{Duality} \label{sub:dual}
We now state and prove duality results for $\eta$-Hardy spaces in the context of orthogonal root systems. Notice that by Theorem \ref{thm:1} the vector 
space of (finite) linear combinations of $(\eta,A)$-atoms and  $(\eta,B)$-atoms is dense in $H^1_\eta(\R_{+,k})$.
%%%%%%%%%%%%%%%%%%%%%%%%%%%%%%%%%%%%
\begin{theorem} \label{thm:3} 
Let $\eta\in\Z_2^k$. The dual of $H^1_\eta(\R_{+,k})$ is $\BMO_\eta(\R_{+,k})$. More precisely, for any $b\in \BMO_\eta(\R_{+,k})$ the formula
	\begin{equation}\label{eq:6}
	L_b(f)=\int_{\R_{+,k}} bf 
	\end{equation}
defined initially on the vector space of linear combinations of ($\eta$,A/B)-atoms, and then uniquely extended to the whole $H^1_\eta(\R_{+,k})$, 
gives a bounded linear functional on $H^1_\eta(\R_{+,k})$. Conversely, every element of $H^1_\eta(\R_{+,k})'$ is of this form. Moreover, 
$\Vert L_b\Vert\simeq \Vert b\Vert_{\BMO_\eta(\R_{+,k})}$.
\end{theorem}
%%%%%%%%%%%%%%%%%%%%%%%%%%%%%%%%%%%%%%	

\begin{proof}
First we note that \eqref{eq:6} is a correct definition: if $f=\sum_{j=1}^n(\a_j a_j+\b_j b_j)$, where each $a_j/b_j$ is an $(\eta,A/B)$-atom supported in 
a cube in $\R_{+,k}$, then  $\int_{\R_{+,k}} |bf|<\infty$. This is because $a_j$ and $b_j$ are square integrable and $b\in L^2_{\rm loc}(\R_{+,k})$ (since 
$\calE^\eta b\in\BMO(\R)$, as a consequence of the John-Nirenberg inequality one has $\calE^\eta b\in L^2_{\rm loc}(\R)$ and hence $ b\in L^2_{\rm loc}(\R_{+,k})$).

Now, fix $b\in\BMO_\eta(\R_{+,k})$. By Theorem \ref{thm:2} there exists $B\in\BMO(\R)$, supported in $\RR_+^{\eta,1}$ and such that $B\big\vert_{\R_{+,k}}=b$. 
Let $f\in H^1_\eta(\R_{+,k})$ be a   linear combination of ($\eta,A/B)$-atoms. By Theorem \ref{thm:1} there exists $F\in H^1(\R)$, supported in $\RR_+^{\eta,0}$ 
and such that $F\big\vert_{\R_{+,k}}=f$.	In fact, see the proof of Theorem \ref{thm:1}, we can select $F=(\calE^\eta f)\ind_{\RR^{\eta,0}_+}$ which has the 
required properties, and moreover satisfies $\Vert F\Vert_{H^1(\R)} \simeq  \Vert f\Vert_{H^1_\eta(\R_{+,k})}$. Notably, $F$ is a   linear combination of 
(classical) atoms supported in $\R_{+,k}$. By the localization of the supports of $B$ and $F$, and by the duality between $H^1(\R)$ and $\BMO(\R)$ we obtain
	\begin{equation*}
	|L_b(f)| =\Big|\int_{\R_{+,k}} bf\Big| = \Big|\int_{\R} BF\Big|\leq  \Vert B\Vert_{\BMO(\R)}\Vert F\Vert_{H^1(\R)}.
	\end{equation*}
Taking the infimum over $B$ and $F$ gives
	$$
	|L_b(f)| \lesssim  \Vert b\Vert_{\BMO_\eta(\R_{+,k})}\Vert f\Vert_{H^1_\eta(\R_{+,k})},
	$$
and this proves that $L_b\in H^1(\R_{+,k})'$ and $\Vert L_b\Vert\lesssim \Vert b\Vert_{\BMO_\eta(\R_{+,k})}$.
	
Conversely, fix $L\in H^1_\eta(\R_{+,k})'$. Let $X$ be a subspace of $H^1(\R)$ composed of the functions vanishing in $\RR^{\eta,1}\setminus \R_{+,k}$. 
Notice that by Theorem \ref{thm:1}, if $F\in X$, then $F\big\vert_{\R_{+,k}} \in H^1_\eta(\R_{+,k})$, and 
$\Vert F\vert_{\R_{+,k}}\Vert_{H^1_\eta(\R_{+,k})}\lesssim \Vert F\Vert_{H^1(\R)}$.  

The dual space of $X$ is isometrically isomorphic to $Y:=\BMO_r((\RR^{\eta,1}\setminus \R_{+,k})^c)$, namely the space of functions (modulo constants) from 
$\BMO(\R)$ restricted to $(\RR^{\eta,1}\setminus \R_{+,k})^c$ with the usual norm: $\Vert b\Vert_Y =\inf_B \Vert B\Vert_{\BMO(\R)}$, where the infimum is taken 
over all extensions $B\in \BMO(\R)$ of $b$ to $\R$. The identification between the functionals from $X'$ and the elements of $Y$ is analogous to \eqref{eq:6}. 
This fact follows from the classical theory of Banach spaces (see \cite[Theorem 10.1, p.\,88]{Co}), because $X$ is a closed subspace of $H^1(\R)$, but also 
from \cite{ART,C}, since $(\RR^{\eta,1}\setminus \R_{+,k})^c$ is a strongly Lipschitz domain in $\R$  (to be precise, we should consider  the interior of 
this set just to have the set open, but this is irrelevant).
	
Let us now define $\varphi:X\to\C$ by	$\varphi(F)=L\big(F\vert_{\R_{+,k}}\big)$. Notice that $\varphi\in X'$ and $\Vert \varphi\Vert \leq \Vert L\Vert$. 
Indeed, this follows from the bound
$$
|\varphi(F)|\leq \Vert L\Vert \big\| F\vert_{\R_{+,k}}\big\|_{H^1_\eta(\R_{+,k})} \lesssim \Vert L\Vert \Vert F\Vert_{H^1(\R)}.
$$
Hence, there exists (a function modulo constants) $B\in Y$ such that
	\begin{equation*}
	\varphi(F) = \int_{\R} FB
	\end{equation*}
for any $F$ being a linear combination of atoms supported in $(\RR^{\eta,1}\setminus \R_{+,k})^c$, and $\Vert \varphi\Vert=\Vert B\Vert_{\BMO(\R)}$.
	
At this point it is appropriate to consider the case $\eta={\bf 0}$ separately. Notice that in this situation $(\RR^{\eta,1}_+\setminus \R_{+,k})^c$ is simply equal to $\R_{+,k}$ (to be precise, equal to $\overline{\R_{+,k}}$, but this is irrelevant). Thus, $Y=\BMO_r(\R_{+,k})$, and the proof of the duality is completed.
	
From now on we assume that $\eta\neq{\bf 0}$. By the definition of $\varphi$, if $F|_{\R_{+,k}}=0$, then $\varphi(F)=0$. Thus, $B$ is constant in 
$(\RR^{\eta,1}_+)^c$. We choose a representative $\widetilde{B}$ from the abstract class of $B$ which vanishes in $(\RR^{\eta,1}_+)^c$. Then, for any 
$f\in H^1_\eta(\R_{+,k})$ and for its any extension $F\in X$ we have
$$
\varphi(F)= \int_{\R_{+,k}} f \widetilde{B}|_{\R_{+,k}} 
$$ 
and $\Vert \varphi\Vert =\inf_{B} \Vert B\Vert_{\BMO(\R)}$, where the infimum is taken over the functions $B\in\BMO(\R)$ such that 
$B|_{\R_{+,k}}=\widetilde{B}|_{\R_{+,k}}$ and $B$ is supported in $\RR^{\eta,1}_+$. By Theorem \ref{thm:2}, $\widetilde{B}|_{\R_{+,k}}\in\BMO_\eta(\R_{+,k})$ 
and $\Vert \varphi\Vert\simeq \Vert \widetilde{B}|_{\R_{+,k}}\Vert_{\BMO_\eta(\R_{+,k})}$. 

This concludes the proof. 
\end{proof}

The local version of Theorem \ref{thm:3} is the following.
%%%%%%%%%%%%%%%%%%%%%%%%%%%%%%%%%%%%
\begin{theorem} \label{thm:3(loc)} 
Let $\eta\in\Z_2^k$. The dual of $h^1_\eta(\R_{+,k})$ is ${\rm bmo}_\eta(\R_{+,k})$. More precisely, for any $b\in {\rm bmo}_\eta(\R_{+,k})$ the formula
	\begin{equation*}%\label{eq:6(loc)}
	L_b(f)=\int_{\R_{+,k}} bf 
	\end{equation*}
defined initially on the vector space of linear combinations of local ($\eta$,A/B)-atoms, and then uniquely extended to the whole 
$h^1_\eta(\R_{+,k})$, gives a bounded linear functional on $h^1_\eta(\R_{+,k})$. Conversely, every element of $h^1_\eta(\R_{+,k})'$ is of this form. Moreover, 
$\Vert L_b\Vert\simeq \Vert b\Vert_{{\rm bmo}_\eta(\R_{+,k})}$.
\end{theorem}
%%%%%%%%%%%%%%%%%%%%%%%%%%%%%%%%%%%%%%	
\begin{proof}
In the first part of the proof we copy, \textit{mutatis mutandis}, the argument from the proof of Theorem \ref{thm:3}. Obviously, in place of
Theorems \ref{thm:1} and \ref{thm:2} we now use their local versions, Theorems \ref{thm:1(loc)} and \ref{thm:2(loc)}. The same local versions are used in 
the second part of the proof with necessary changes.
\end{proof}

\begin{remark} \label{rem:fin}
In the general case of a strongly Lipschitz domain $\Omega$ it was proved  in  \cite[Theorem 5 (b), (d)]{AR}-arXiv, that the duals 
of $H^1_z(\Omega)$ and $H^1_r(\Omega)$ are $\BMO_r(\Omega)$ and $\BMO_z(\Omega)$, respectively. Therefore, in the general setting, since $C_+$ is a special 
Lipschitz domain, it follows that for  $\eta={\rm triv}$ or  $\eta={\rm sgn}$ we have: the dual of $H^1_{\rm triv}(C_+)$ is $\BMO_r(C_+)$, the dual of 
$H^1_{\sgn}(C_+)$ is $\BMO_z(C_+)$. 
\end{remark}

The proposition concluding this section  establishes relations between the spaces $\BMO^1_\eta(\R_{+,k})$ for different $\eta$'s (cf. also 
Corollary \ref{rem:com_bmo}) with analogous statement in the local case, and is, in some sense, dual to Proposition \ref{prop:rel}.
%%%%%%%%%%%%%%%%%%%%%%%%%%%%%%%%%%%%%%%
\begin{proposition} \label{prop:relBMO}
If $\eta^{(1)}\le\eta^{(2)}$ (the lexicographical order), then $\BMO_{\eta^{(2)}}(\R_{+,k})\hookrightarrow \BMO_{\eta^{(1)}}(\R_{+,k})$. 
In particular, we have 
$$
\BMO_{\textbf{1}}(\R_{+,k})\hookrightarrow \BMO_\eta(\R_{+,k})\hookrightarrow \BMO_{\textbf{0}}(\R_{+,k}).
$$
Analogous statements hold for the local spaces ${\rm bmo}_{\eta}(\R_{+,k})$. 
\end{proposition}
%%%%%%%%%%%%%%%%%%%%%%%%%%%%%%%%%%%%%%%
\begin{proof}
If $\eta^{(1)}\le\eta^{(2)}$, then  $\RR^{\eta^{(2)},1}_+\subset \RR^{\eta^{(1)},1}_+$. Therefore the claimed continuous embedding is an obvious 
consequence of the equivalence $(1)\Leftrightarrow(3)$ in Theorem \ref{thm:2}. For the local case we repeat the same argument using 
Theorem \ref{thm:2(loc)}. 
\end{proof}


\begin{thebibliography}{99}


%\baselineskip=17pt
\bibitem{AR} P. Auscher, E. Russ, Hardy spaces and divergence operators on strongly Lipschitz domains of $R^n$, 
J. Funct. Anal. 201 (2003), 148--184.

\bibitem{ART} P. Auscher, E. Russ, P. Tchamitchian, Hardy Sobolev spaces on strongly Lipschitz domains of $R^n$, 
J. Funct. Anal. 218 (2005), 54--109.

\bibitem{C}  D.-C. Chang, The dual of Hardy spaces on a bounded domain in $\mathbb R^n$, Forum Math. 6 (1994), 65--81.

\bibitem{CKS} D.-C. Chang, S.G. Krantz, E.M. Stein, $H^p$ theory on a smooth domain in $R^N$ and elliptic boundary value problems, 
J. Funct. Anal. 114 (1993), 286--347.

\bibitem{CDS}  D.-C. Chang, G. Dafni,  E.M. Stein, Hardy spaces, BMO, and boundary value problems for the Laplacian on a smooth domain, 
Trans. Amer. Math. Soc. 351 (1999), 1605--1661.

\bibitem{CW} R. R. Coifman, G. Weiss, Extensions of Hardy spaces and their use in analysis, Bull. Amer. Math. Soc.  83 (1977), 569--645.

\bibitem{Co} J. B. Conway, A Course in Functional Analysis, Springer, Second Edition, Springer, 1990. 

\bibitem{CMT} M. Costabel, A. McIntosh, R. Taggart, Potential maps, Hardy spaces, and tent spaces on special Lipschitz domains, Publ. Mat. 57 (2013), 295--331.

\bibitem{DDSY} D. Deng, X. T. Duong, A. Sikora, L. Yan, Comparison of the classical BMO with the BMO spaces associated with operators and applications, Rev. Mat. Iberoam. 24 (2008), 267--296.

\bibitem{DY} X. T. Duong, L. Yan, Duality of Hardy and BMO spaces associated with operators with heat kernel bounds, Trans. Amer. Math. Soc. 18 (2005), 943--973.

\bibitem{DX} C. Dunkl, Yuan Xu, Orthogonal polynomials of several variables, Encyclopedia Math. Appl. 81, Cambridge Univ. Press, 2001. 

\bibitem{G-CRdF} J. Garci\'a-Cuerva, J.L. Rubio de Francia, Weighted Norm Inequalities and Related Topics, North-Holland, Amsterdam, 1985.

\bibitem{Hu} J. E. Humphreys, Reflection groups and Coxeter groups, Cambridge Univ. Press, 1990.

\bibitem{JSW} A. Jonsson, P. Sj\"ogren, H. Wallin, Hardy and Lipschitz spaces on subsets of $\mathbb R^n$, Studia Math.  
80 (1984), 141--166.

\bibitem{Ka} R. Kane, Reflection groups and Invariant Theory, CMS Books in Mathematics, Springer,  2001. 

\bibitem{MS} J. Ma\l{}ecki, K. Stempak, Reflection principles for functions of Neumann and Dirichlet Laplacians on open reflection invariant subsets of $\R$,   
Studia Math. 251 (2020), 171--193. 

\bibitem{M} A. Miyachi, $H^p$ spaces over open subsets of $\mathbb R^n$, Studia Math.  95 (1990), 205--228.

\bibitem{Sem}  S. Semmes, A primer on Hardy spaces, and some remarks on a theorem of Evans and M\"uller, Comm. Partial Diff. Eq., 19 (1994), 277--319.

\bibitem{Stein}  E. M. Stein, Singular integrals and differentiability properties of functions,  Princeton Univ. Press, 1970.

\bibitem{St}  E. M. Stein, Harmonic Analysis: Real-Variable Methods, Orthogonality, and Oscillatory Integrals, 1st ed., Princeton Univ. Press, 1993.

\bibitem{S1}  K. Stempak, Finite reflection groups and symmetric extensions of Laplacian,  Studia Math. 261 (2021), 241--267. 

\bibitem{S2}  K. Stempak, The Laplacian with mixed Dirichlet-Neumann boundary conditions on Weyl chambers, J. Diff. Equations (2022), 348--370. 

\end{thebibliography}
\end{document}